\documentclass[11pt, reqno]{amsart}
\usepackage[dvipsnames]{xcolor}
\usepackage{amscd, amsmath, amssymb, amsthm, yfonts, mathrsfs, hhline, tikz}
\usetikzlibrary{arrows, decorations.markings, patterns}
\tikzset{>=latex}
\usepackage[centering,margin=1.2in,footskip=1cm]{geometry}
\usepackage[matrix,arrow,curve]{xy}

\usepackage{tikz-cd}

\global\binoppenalty=10000

\newtheorem{theorem}{Theorem}[section]
\newtheorem{lemma}[theorem]{Lemma}
\newtheorem{prop}[theorem]{Proposition}
\newtheorem{cor}[theorem]{Corollary}

\theoremstyle{definition}
\newtheorem{defn}[theorem]{Definition}
\newtheorem{remark}[theorem]{Remark}

\numberwithin{equation}{section}

\def\beq{\begin{equation}}
\def\eeq{\end{equation}}

\def\longra{\longrightarrow}

\newcommand{\hr}[1]{\left(#1\right)}                                                    
\newcommand{\hm}[1]{\left|#1\right|}                                                    
\newcommand{\ha}[1]{\left\langle#1\right\rangle}                                        
\newcommand{\hs}[1]{\left[#1\right]}                                                    
\newcommand{\hc}[1]{\left\{#1\right\}}                                                  

\def\le{\leqslant}
\def\ge{\geqslant}

\def\Ac{\mathcal A}
\def\Ad{\operatorname{Ad}}

\def\bs{\boldsymbol}

\def\C{\mathbb C}

\def\Dc{\mathcal D}

\def\eps{\varepsilon}

\def\Fc{\mathcal F}
\def\Frac{\operatorname{Frac}}

\def\Gc{\mathcal G}

\def\Hbb{\mathbb H}
\def\Hc{\mathcal H}

\def\ibf{\mathbf i}

\def\la{\lambda}
\def\La{\Lambda}
\def\Lbb{\mathbb L}

\def\Mc{\mathcal{M}}
\def\Mgt{\mathfrak{M}}

\def\Oc{\mathcal O}

\def\Qbb{\mathbb Q}

\def\R{\mathbb R}

\def\Tc{\mathcal T}

\def\Sc{\mathcal S}
\def\SH{\mathbb{SH}}
\def\Spec{\operatorname{Spec}}

\def\Vc{\mathcal V}

\def\wt{\operatorname{wt}}

\def\WGL{\widetilde{GL}}

\def\Xc{\mathcal X}

\def\Z{\mathbb Z}

\begin{document}

\title{Cluster structure on genus 2 spherical DAHA: \\
seven-colored flower}
\author{Semeon Arthamonov, Leonid Chekhov, Philippe Di Francesco, Rinat Kedem, Gus Schrader, Alexander Shapiro, Michael Shapiro}

\maketitle

\begin{abstract}
We construct an embedding of the Arthamonov-Shakirov algebra of genus 2 knot operators into the quantized coordinate ring of the cluster Poisson variety of exceptional finite mutation type $X_7$. The embedding is equivariant with respect to the action of the mapping class group of the closed surface of genus 2. The cluster realization of the mapping class group action leads to a formula for the coefficient of each monomial in the genus 2 Macdonald polynomial of type $A_1$ as sum over lattice points in a convex polyhedron in 7-dimensional space.

\bigskip
\bigskip

\leftskip 24.2em
\noindent Seven-colored flower, glide, \\
Cross the skies from side to side. \\
West to east, then south, turn north, \\
Completing circles, forth and forth. \\
Once you kiss the earth, comply, \\
Grant my wish, let dreams fly high.
\vspace{-5pt}
\begin{flushright}
{\scriptsize V. Kataev \\ Translation by ChatGPT}
\end{flushright}

\leftskip 0em

\end{abstract}

\section{Introduction}

The Double Affine Hecke Algebra (DAHA) is an associative $\mathbb{Q}(q,t)$-algebra introduced and studied by I.Cherednik.
It is closely connected with the topology of the once-punctured torus, and the mapping class group $SL(2,\mathbb{Z})$ of the latter acts on DAHA by outer automorphisms.

The DAHA has an important subalgebra called the spherical subalgebra, which in the case of the $A_1$ root system can identified with the algebra generated by the operators
$$
\Oc_A=\frac{tx-t^{-1}x^{-1}}{x-x^{-1}} T_x + \frac{tx^{-1}-t^{-1}x}{x^{-1}-x} T_x^{-1}
\qquad  \text{and} \qquad
\Oc_B=x+x^{-1}
$$
acting on the space of symmetric Laurent polynomials in the variable $x$. The operator $T_x$ acts as a multiplicative shift in the variable $x$, namely, $(T_x f)(x)=f(qx)$. These operators are associated with the $A$- and $B$- cycles respectively on the punctured torus. The operator $\Oc_A$ coincides with the Macdonald difference operator $M$ associated to the root system $A_1$, and its complete set of eigenfunctions $\{P_l\}_{l\in\mathbb{Z}_{\ge0}}$ in the space of symmetric Laurent polynomials are the $A_1$ Macdonald polynomials. On the other hand, the $A_1$ spherical DAHA is naturally embedded into the universally Laurent algebra $\widehat \Lbb^q_{tor}$, which quantizes the coordinate ring of the moduli space of framed $SL_2$-local systems on the punctured torus. In section~\ref{sec:Macdo-gen1} we illustrate the use of cluster structure on $A_1$ spherical DAHA by expressing Macdonald polynomials in term of the Whittaker ones.

In \cite{AS19}, Arthamonov and Shakirov proposed a genus 2 generalization of the $A_1$-spherical DAHA.
More specifically (see Section~\ref{sec:g2-defs}), they found a system of six operators acting on the ring of Laurent polynomials in three variables $(x_{12},x_{13},x_{23})$: three commuting operators $\Oc_{B_{ij}}$ of multiplication by $x_{ij}+x_{ij}^{-1}$, along with three commuting finite difference operators $\mathcal{O}_{A_k}$ which were shown to admit a basis of eigenfunctions $\Phi_{\bs l}$ labelled by certain \emph{admissible triples} $\bs l = (l_1,l_2,l_3)\in\mathbb{Z}^3$. When the parameter $\bs l\in \mathbb{Z}^3$ lies on certain special rays, the genus 2 Macdonald polynomials $\Phi_{\bs l}$ reduce to multiples of their genus 1 counterparts: for example, we have $\Phi_{l,l,0} = c_lP_l(x_{12})$ where $c_l\in\mathbb{Q}(q,t)$ is an explicit $l$-dependent scalar -- again, see Section~\ref{sec:g2-defs} for details. In subsequent work~\cite{CS21}, the topological meaning of the Arthamonov-Shakirov algebra was further clarified: a specialization at $t=q$ was shown to recover the Kaufmann bracket skein algebra of a closed genus 2 surface. In what follows, we denote the latter by $\Sigma_{2,0}$.

In this manuscript we obtain a cluster-algebraic realization of the Arthamonov-Shakirov algebra analogous to the one described above for the spherical DAHA. The role of the Fock-Goncharov moduli space $\widehat\Lbb^q_{tor}$ is played by a 1-parametric deformation of the quantized ring of functions on the Teichm\"uller space for closed genus two Riemann surfaces, which was shown in \cite{CS23} to support a cluster structure of exceptional finite mutation type $X_7$. 

The main idea of our construction is to interpret the generators $\Oc_{B_{ij}} = x_{ij} + x^{-1}_{ij}$ as the eigenvalues of the quantum Teichm\"uller geodesic length operators associated to the pants decomposition of $\Sigma_{2,0}$, obtained by cutting along the three simple closed curves $(B_{12},B_{13},B_{23})$, see~Figure~\ref{fig:genus2curve}. The dual operators $\Oc_{A_i}$ are then recovered by expressing the $A$-cycle geodesic length operators in the basis of eigenvectors for the $B$-cycle ones. 
The key to carrying this out is understanding the local picture of the length operators for all (open and closed) curves in a cylinder containing one of the cutting curves, which we treat in detail in Section~\ref{sec:cylinder}.

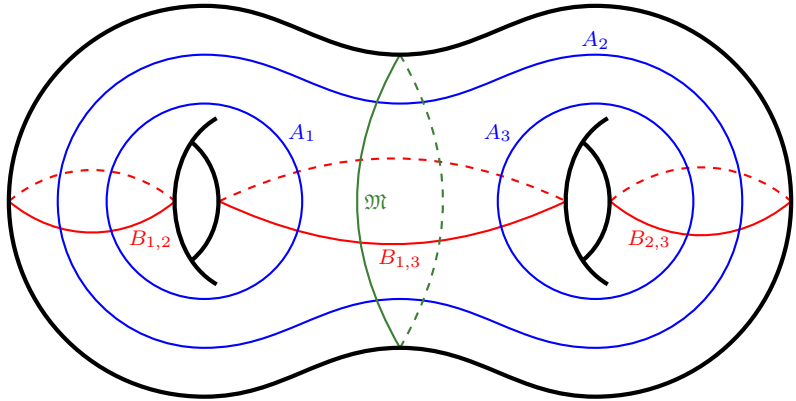
\begin{figure}[b]
\begin{tikzpicture}[every node/.style={inner sep=0.3, minimum size=0.5cm, circle, font=\scriptsize}, thick, x=0.65cm, y=0.65cm]

\draw[red] (-8,0) to[out=-40, in=-140] (-4.6,0);
\draw[red, dashed] (-8,0) to[out=40, in=140] (-4.6,0);

\draw[red] (-3.7,0) to[out=-25, in=-155] (3.4,0);
\draw[red, dashed] (-3.7,0) to[out=25, in=155] (3.4,0);

\draw[red] (4.3,0) to[out=-40, in=-140] (8,0);
\draw[red, dashed] (4.3,0) to[out=40, in=140] (8,0);

\draw[blue] (-4,0) circle (2);
\draw[blue] (4,0) circle (2);

\draw[blue]
(-7,0) to[out=90, in=180] (-4,3) to[out=0, in=180] (0,2) to[out=0, in=180] (4,3) to[out=0, in=90]
(7,0) to[out=-90,in=0] (4,-3) to[out=180,in=0] (0,-2) to[out=180,in=0] (-4,-3) to[out=180,in=-90] (-7,0);

\draw[OliveGreen] (0,3) to[out=-120, in=120] (0,-3);
\draw[OliveGreen, dashed] (0,3) to[out=-60, in=60] (0,-3);

\draw[ultra thick]
(-8,0) to[out=90, in=180] (-4,4) to[out=0, in=180] (0,3) to[out=0, in=180] (4,4) to[out=0, in=90]
(8,0) to[out=-90,in=0] (4,-4) to[out=180,in=0] (0,-3) to[out=180,in=0] (-4,-4) to[out=180,in=-90] (-8,0);

\draw[ultra thick] (-3.75,1.7) to[out=-150, in=150] (-3.75,-1.7);
\draw[ultra thick] (-4.25,1.2) to[out=-40, in=40] (-4.25,-1.2);

\draw[ultra thick] (4.25,1.7) to[out=-150, in=150] (4.25,-1.7);
\draw[ultra thick] (3.75,1.2) to[out=-40, in=40] (3.75,-1.2);

\node[color=red] at (-5.1,-.8) {$B_{1,2}$};
\node[color=red] at (0,-1.2) {$B_{1,3}$};
\node[color=red] at (5.1,-.8) {$B_{2,3}$};

\node[color=blue] at (-2,1.4) {$A_1$};
\node[color=blue] at (4,3.3) {$A_2$};
\node[color=blue] at (2,1.4) {$A_3$};

\node[color=OliveGreen] at (-.5,0) {$\Mgt$};

\end{tikzpicture}
\caption{Surface $\Sigma_{2,0}$ with a separating cycle $\Mgt$ and the non-separating cycles $A_1, A_2, A_3$ and $B_{1,2}, B_{1,3}, B_{2,3}$.}
\label{fig:genus2curve}
\end{figure}

Our main result, Theorem~\ref{thm:g2-embed}, is the construction of a geometrically natural embedding of the Arthamonov-Shakirov algebra into (a cover of) the universally Laurent algebra of type $X_7$. This embedding is equivariant under the action of the mapping class group $\Gamma_{2,0}$ of the surface $\Sigma_{2,0}$. This in turn allows us to obtain our second main result in Theorem~\ref{thm:g2-eigen}: a non-recursive formula for the coefficients of the genus 2 Macdonald polynomials $\Phi_{\bs l}$ as weighted sums over lattice points in certain convex polyhedron in $\mathbb{R}^7$. The cluster realization also allows one to consider an analytic theory of representations of the Arthamonov-Shakirov algebra, which we comment briefly on in Section~\ref{sec:analytic}. Finally, in Section~\ref{sec:class} we relate the quasi-classical limit of our constructions to the main results of~\cite{CS23}.

\subsection*{Acknowledgements}

P.D.F. is supported by the Morris and Gertrude Fine Endowment and the Simons Foundation Grant MP-TSM-00002262. R.K. acknowledges support from the Simons Foundation Grant MP-TSM-0000194. G.S. has been supported by the NSF Standard Grant DMS-2302624. M.S. has been supported by the NSF research grant DMS-2100791. A.S. has been supported by the European Research Council under the European Union’s Horizon 2020 research and innovation programme under grant agreement No 948885 and by the Royal Society University Research Fellowship.

\section{$SL_2$ Macdonald polynomials in genera 1 and 2}

In this section we recall the basics about $SL_2$ Macdonald polynomials, and then present similar statements for their genus 2 analogues, following~\cite{AS19}.

\subsection{Genus 0}
\label{sec:genus-0}
We start by considering an algebra of difference operators which we will later see is naturally associated to curves on a cylinder -- see Section~\ref{sec:cylinder} for the topological explanation of the formulas to follow. For $n\in\mathbb{Z}$, let us define $q$-difference operators $\check H_n$ by
$$
q^{-\frac{n}{2}} \check H_n = \frac{x^{n}}{1-x^2} T_x + \frac{x^{-n}}{1-x^{-2}} T_x^{-1}.
$$
The operators $\check H_n$ act on the space $\Sc_{q,t}$ of $\mathbb{Q}(q,t)$-valued Laurent polynomials in $x$, symmetric under the involution $x \mapsto x^{-1}$. Let us now focus on the \emph{dual Toda Hamiltonian} $\check H_0$. It acts diagonalizably with distinct eigenvalues on the space $\Sc_{q,t}$,  and a basis of eigenvectors is given by the \emph{Whittaker polynomials} $W_l(x) := W_l(x;q^2)$:
\beq
\label{eq:Whit-poly}
W_l(x;q^2) = \sum_{k=0}^l \binom{l}{k}_{q^2} x^{l-2k},
\eeq
with eigenvalues
\beq
\label{eq:Whit-dual-eigen-prop}
\check H_0 W_l(x)  = q^{-l} W_l(x).
\eeq
Here the $q$-binomial coefficient is defined as
$$
\binom{n}{k}_q = \frac{(q;q)_n}{(q;q)_k(q;q)_{n-k}},
$$
where $(X;q)_n$ is the standard notation for the $q$-Pochhammer symbol
$$
(X;q)_n = \prod_{k=0}^{n-1}(1-q^kX).
$$
The \emph{Pieri rule} for the Whittaker polynomials describes the expansion of $(x+x^{-1}) W_l(x)$ in the basis $\hc{W_l}_{l\in\Z}$:
\beq
\label{eq:Whit-Pieri}
\hr{x+x^{-1}} W_l(x) = W_{l+1}(x) + \hr{1-q^{2l}} W_{l-1}(x).
\eeq
Writing $L$ for the operator of multiplication by $x+x^{-1}$, we note that
$$
q^{\pm\frac{1}{2}}\check H_n L - q^{\mp\frac{1}{2}}L\check H_n = (q^{\pm1}-q^{\mp1})\check H_{n\pm1}.
$$
In what follows, we denote by $\SH_{g=0}$ the subalgebra of the ring of difference operators generated by $L$ and $\check H_0$. Note that the algebra $\SH_{g=0}$ carries an action of the mapping class group of a cylinder, which is isomorphic to $\Z$, with the generator $\tau$ acting by translation:
$$
\tau(L) = L \qquad\text{and}\qquad \tau(\check H_n) = \check H_{n+1}.
$$

\subsection{Genus 1}

Recall that the $SL_2$ Macdonald operator is the $q$-difference operator $M$, denoted $\Oc_A$ in the introduction, and defined by
\beq
\label{eq:Macdo-op}
M = \frac{tx-t^{-1}x^{-1}}{x-x^{-1}} T_x + \frac{tx^{-1}-t^{-1}x}{x^{-1}-x} T_x^{-1},
\eeq
where $T_x f(x) = f(qx)$. Let us point out that the Macdonald operator can be written in terms of the operators $\check H_n$ as
\beq
\label{eq:Macdo-dual-Toda}
M = t^{-1} \check H_0 - q^{-1}t \check H_2,
\eeq
and thus the dual Toda Hamiltonian $\check H_0$ is recovered as the Whittaker limit of the Macdonald operator:
$$
\check H_0 = (tM)|_{t=0}.
$$
The action of $M$ on $\Sc_{q,t}$ is diagonalizable, with distinct eigenvalues: an eigenbasis is given by the symmetric $SL_2$ \emph{Macdonald polynomials} $P_l(x) = P_l(x; t^2,q^2)$ as $l$ runs over $\Z_{\ge0}$. 
The corresponding eigenvalues for the finite difference operator $M$ are given  by
$$
M \cdot P_l(x) = \hr{q^{l}t + q^{-l}t^{-1}} P_l(x),
$$
and the polynomials $P_l$ can be expressed in terms of terminating $_2\psi_1$ basic hypergeometric series:
\begin{align*}
P_l(x; t^2,q^2) &= \sum_{r=0}^l \frac{\hr{q^{2l};q^{-2}}_r \hr{t^2;q^2}_r}{\hr{q^{2(l-1)t^2};q^{-2}}_r \hr{q^2;q^2}_r} x^{l-2r}\\
&=x^{-l} \cdot {}_2\psi_1\hr{q^{-2l},t^2;q^{2(1-l)}t^{-2};x^2q^2t^{-2} \,\Big|\, q^2},
\end{align*}
where
$$
_2\psi_1(a,b;c;z \,|\, q) = \sum_{n\ge0}\frac{(a;q)_n(b;q)_n}{(c;q)_n(q;q)_n}z^n.
$$
We also recall Heine's $q$-analogue of Gauss' summation formula
$$
_2\psi_1\left(a,b;c;\frac{c}{ab} \,\Big|\, q\right) = \frac{(\frac{c}{a};q)_\infty (\frac{c}{b};q)_\infty}{(c;q)_\infty (\frac{c}{ab};q)_\infty},
$$
whose right hand side terminates in the special case $b = q^{-n}$, $n\in\mathbb{Z}_{\ge0}$, and reduces to the Chu--Vandermonde formula
\beq
\label{eq:gauss-thm}
_2\psi_1\left(a,q^{-n};c;q^n\frac{c}{a} \,\Big|\, q\right) = \frac{(c/a;q)_n}{(c;q)_n}.
\eeq

The Pieri rule for Macdonald polynomials takes the form
\beq
\label{eq:Pieri-1}
\hr{x+x^{-1}}P_l(x) = P_{l+1}(x) + \frac{\big(1-q^{2l}\big)\big(1-q^{2(l-1)}t^4\big)}{\big(1-q^{2l}t^2\big)\big(1-q^{2(l-1)}t^2\big)} P_{l-1}(x).
\eeq
Note that one may instead define the polynomials $P_l(x)$ by fixing the initial conditions $P_l(x) = 0$ for $l<0$, $P_0(x)=1$, and iterating the Pieri rule~\eqref{eq:Pieri-1} to compute $P_l(x)$ for $l>0$.

The $GL_n$ double affine Hecke algebra (DAHA) $\Hbb_{q,t}$ is a quotient of the $\mathbb{Q}(q,t)$-group algebra of the elliptic braid group by certain quadratic Hecke relations. It contains an idempotent $e \in \Hbb_{t,q}$ and a \emph{spherical} subalgebra $\mathbb{SH}_{q,t}=e\Hbb_{q,t}e$. Here we will skip the precise definition of DAHA and instead use the following facts. First, the spherical DAHA admits a faithful representation on $\mathbb{Q}(q,t)[T^W]$, where $T$ and $W$ are respectively the maximal torus and the Weyl group of $GL_n$. The elements of the spherical DAHA act in this representation by finite difference operators,  see~\cite{Che05}. Second, the algebra $\mathbb{SH}_{q,t}$ contains elements $E_v$, labelled by primitive vectors $v \in \Z^2$, and is generated over $\mathbb{Q}(q,t)$ by $E_{(\pm1,0)}$ and $E_{(0,\pm1)}$, see~\cite{SV11,BS12}. In Cherednik's representation the elements $E_{(0,1)}$ and $E_{(1,0)}$ act by the Macdonald operator and the operator of multiplication by the first elementary symmetric function respectively. Third, the algebra $\SH_{q,t}$ carries an action of the mapping class group of a torus, which is isomorphic to $SL(2,\Z)$. An element $g \in SL(2,\Z)$ acts on $\SH_{q,t}$ in such a way that $g \cdot E_v = E_{g \cdot v}$. Analogous constructions of DAHA can be carried out for the groups $G=SL_n,PGL_n$; see~\cite{Che05}  for further details.

In the case  $G=SL_2$ the situation is especially simple: the spherical DAHA is isomorphic to the subalgebra of symmetric $q$-difference operators in a single (invertible) variable $x$ generated by the operator~\eqref{eq:Macdo-op} along with the operator of multiplication by $x+x^{-1}$. In what follows, we denote the $SL_2$ spherical DAHA by $\SH_{g=1}$.

\subsection{Genus 2} 
\label{sec:g2-defs}
Consider the ring $\mathbb{C}(q,t)\hs{x_{12}^{\pm1},x_{13}^{\pm1},x_{23}^{\pm1}}$ of Laurent polynomials in three variables $x_{ij}$ with $i < j$. In~\cite{AS19}, the authors introduced the triple $\Oc_{A_k}$, $k=1,2,3$ of commuting $q$-difference operators on this ring. The operator $\Oc_{A_1}$ is defined to be
$$
\Oc_{A_1} = \sum_{a,b \in \hc{\pm1}} ab \frac{(1-tx_{12}^a x_{13}^b x_{23} ) (1-tx_{12}^a x_{13}^b x_{23}^{-1} )}{tx_{12}^a x_{13}^b  (x_{12}-x_{12}^{-1}) (x_{13}-x_{13}^{-1})} T_{{12}}^a T_{{13}}^b,
$$
and is symmetric under the permutation of the indices 2 and 3. The operators $\Oc_{A_2} ,\Oc_{A_3} $ are obtained by applying permutations of $\hc{1,2,3}$ to the indices in the formula above. The $q$-difference operators $\Oc_{A_k}$ preserve the subring $\Sc^{\otimes 3}_{q,t}$ consisting of Laurent polynomials symmetric under the action of $(\Z/2\Z)^3$ generated by the involutions  $x_{ij} \mapsto x_{ij}^{-1}$ for $1\le i<j\le 3$. Denote by
$$
\Oc_{B_{ij}} = x_{ij} + x_{ij}^{-1}
$$
the operator of multiplication by $x_{ij} + x_{ij}^{-1}$, where $1 \le i < j \le 3$. The \emph{genus 2 spherical DAHA}, which we denote $\SH_{g=2}$, was defined in~\cite{AS19}, as the subalgebra of $q$-difference operators in variables $(x_{12}, x_{13}, x_{23})$, generated by operators $\Oc_{A_k}$ and $\Oc_{B_{ij}}$.

Collecting coefficients in $t$, we can express the operators $\Oc_{A_k}$ in terms of the single variable difference operators $\check H_n^{(ij)}$ acting in the variable $x_{ij}$ as
\beq
\label{eq:collect-coeffs}
\Oc_{A_1} = t^{-1} \check H_0^{(12)} \check H_0^{(13)} - \hr{x_{23}+x_{23}^{-1}} \check H_1^{(12)} \check H_1^{(13)} + t \check H_2^{(12)} \check H_2^{(13)},
\eeq
with the other $\Oc_{A_i}$ obtained by permutation of indices. In particular, we see that the algebra $\mathbb{SH}_{g=2}$ is contained in the tensor product $\mathbb{SH}_{g=0}^{\otimes 3}$ of three copies of the algebra of genus 0 difference operators.

In~\cite{AS19}, the genus 2 Macdonald polynomials $\Phi_{\bs l}(\bs x) = \Phi_{\bs l}(\bs x; t,q)$ were then defined using their Pieri rules. Let us call a triple $\bs l \in \Z^3$ \emph{admissible} if $\bs l \in \Z_{\ge 0}^3$, $\underline{\bs l} \in 2\Z$, and $\bs l$ satisfies the triangle inequalities, i.e.
 $$
l_1\le l_2+l_3, \quad  l_2\le l_1+l_3, \quad l_3\le l_1+l_2.
 $$ 
Fix the initial data $\Phi_{(0,0,0)} = 1$, and $\Phi_{\bs l}(\bs x)=0$ unless $\bs l$ is admissible. Then the remaining $\Phi_{\bs l}(\bs x)$ for admissible triples $\bs l$ are characterized by the \emph{genus 2 Pieri rules}, which are obtained as all index-permutations of the following one for multiplication by $x_{12}+x^{-1}_{12}$ :
$$
\hr{x_{12}+x_{12}^{-1}} \Phi_{\bs l}(\bs x) = \sum_{a,b \in \hc{\pm1}} C_{a,b}(\bs l) V_1^{-a} V_2^{-b} \Phi_{\bs l}(\bs x).
$$
 Here
$$
C_{a,b}(\bs l) = ab \dfrac{\left[\frac{a l_1+b l_2+l_3}{2} ,\frac{a+b+2}{2}\right]_{q,t} \left[\frac{a l_1+b l_2-l_3}{2} ,\frac{a+b}{2}\right]_{q,t} \Big[l_1-1,2\Big]_{q,t} \Big[l_2-1,2\Big]_{q,t}}{\Big[l_1,\frac{a+3}{2}\Big]_{q,t}\Big[l_1-1,\frac{a+3}{2}\Big]_{q,t}\Big[l_2,\frac{b+3}{2}\Big]_{q,t} \Big[l_2-1,\frac{b+3}{2}\Big]_{q,t}}
$$
with
$$
[n,m]_{q,t}=\dfrac{q^nt^{m}-q^{-n}t^{-m}}{q-q^{-1}},
$$
and the operators $V_k$ act on $\C(q,t)$-valued functions on $\Z^3$ by shifting the argument:
$$
V_k f(\bs l) = f(\bs l - \bs{\delta^k}),
$$
where
$$
\bs{\delta^1} = (1,0,0), \qquad \bs{\delta^2} = (0,1,0), \qquad \bs{\delta^3} = (0,0,1).
$$
As was shown in~\cite{AS19}, the genus 2 Macdonald polynomials are well-defined, non-zero for all admissible triples $\bs l$, and form an eigenbasis for the action of the genus 2 Macdonald difference operators on $\Sc_{q,t}^{\otimes 3}$. The corresponding eigenvalues are 
$$
\Oc_{A_k} \Phi_{\bs l}(\bs x) = \hr{tq^{l_k} + t^{-1}q^{-l_k}} \Phi_{\bs l}(\bs x).
$$
The relation between the genus 1 and genus 2 Macdonald polynomials is given by the following formulas:
$$
\begin{aligned}
\Phi_{l,l,0}(x_{12},x_{13},x_{23}) = c_l P_l(x_{12}), \\
\Phi_{l,0,l}(x_{12},x_{13},x_{23}) = c_l P_l(x_{13}), \\
\Phi_{0,l,l}(x_{12},x_{13},x_{23}) = c_l P_l(x_{23}),
\end{aligned}
$$
where
$$
c_l = P_l(t) = t^{-\frac{l}{2}} \frac{(t^2;q)_l}{(t;q)_l}.
$$

The interpretation of the algebra $\SH_{g=2}$ as a genus 2 analogue of spherical DAHA is further justified by the existence of an action of $\Gamma_{2,0}$ by automorphisms of $\SH_{g=2}$. Let $a_k$ with $1 \le k \le 3$ and $b_{ij}$ with $1 \le i<j \le 3$ be the Dehn twists along the $A$- and $B$-cycles respectively, as shown on Figure~\ref{fig:genus2curve}. Then the group $\Gamma_{2,0}$ is generated by the elements $a_k$, $b_{ij}$ and the following formulas define its action on $\SH_{g=2}$, see~\cite{AS19}:
\beq
\label{eq:twist-a}
a_k^{\pm1}(\Oc_{B_{ij}}) =
\begin{cases}
\pm(q-q^{-1})^{-1}\hr{q^{\frac{1}{2}}\Oc_{B_{ij}}\Oc_{A_k} - q^{-\frac{1}{2}}\Oc_{A_k}\Oc_{B_{ij}}} \quad &k\in\{i,j\}, \\
\Oc_{B_{ij}}\quad &k\notin\{i,j\}, 
\end{cases}
\eeq
\beq
\label{eq:twist-b}
b_{ij}^{\pm1}(\Oc_{A_k}) =
\begin{cases}
\pm(q-q^{-1})^{-1}\hr{q^{\frac{1}{2}}\Oc_{A_k}\Oc_{B_{ij}} - q^{-\frac{1}{2}}\Oc_{B_{ij}}\Oc_{A_k}} \quad &k\in\{i,j\}, \\
\Oc_{A_k}\quad &k\notin\{i,j\}, 
\end{cases}
\eeq
along with
\beq
\label{eq:twists-ab}
a_k^{\pm1}\hr{\Oc_{A_j}} = A_j \qquad\text{and}\qquad b_{ij}^{\pm1}\hr{\Oc_{B_{kl}}} = \Oc_{B_{kl}}.
\eeq
Further justification for the name comes from the subsequent work~\cite{CS21}, where it was shown that the $t=q$ specialization of $\SH_{g=2}$ is isomorphic to the skein algebra of $\Sigma_{2,0}$. In what follows, will exhibit a quantum cluster structure on  $\mathbb{SH}_{g=2}$, whose classical limit recovers the cluster structure on the Teichm\"uller space of closed genus 2 Riemann surfaces discovered in [CS23].

\section{Quantum cluster varieties}
\label{sec:qcv}

In this section we review the definition of quantum cluster varieties. For more details on the subject, we refer the reader to the foundational paper~\cite{FG09}.

\subsection{Cluster $\Xc$-varieties}
In what follows, we will only work with skew-symmetric quantum cluster varieties with integer-valued forms and no frozen variables, which we incorporate into the definition of a seed.

\begin{defn}
A \emph{seed} is a datum $\Theta=\hr{I, \La, (\cdot,\cdot),\hc{e_i}}$ where
\begin{itemize}
\item $I$ is a finite set;
\item $\La$ is a lattice;
\item $(\cdot,\cdot)$ is a skew-symmetric $\Z$-valued form on $\La$;
\item $\hc{e_i \,|\, i \in I}$ is a basis for the lattice $\La$.
\end{itemize}
Note that the data of the last point is equivalent to that of an isomorphism $\bs e \colon \Z^{I} \simeq \La$. In particular, given a pair of seeds $(\Theta,\Theta')$ with the same index set $I$, we get a canonical isomorphism of abelian groups, but not necessarily an isometry of lattices,
$$
\bs{e'} \circ \bs e^{-1} \colon \La \simeq \La'.
$$
\end{defn}

\begin{defn}
 We say that $(\Theta,\Theta')$ are \emph{equivalent} if the isomorphism $\bs{e'} \circ \bs e^{-1} \colon \La \simeq \La'$ is in fact an isometry, that is $(e_i,e_j)_{\La} = (e'_i, e'_j)_{\La'}$ for all $i,j \in I$. We define a \emph{quiver} to be an equivalence class of seeds.
\end{defn}

The quiver $Q$ associated to a seed $\Theta$ can be visualized as a directed graph with vertices labelled by the set $I$ and arrows given by the adjacency matrix $\eps = \hr{\eps_{ij}}$, where $\eps_{ij} = (e_i,e_j)$. If $\Theta,\Theta'$ are two seeds with nondegenerate skew forms representing the same quiver, then we get a canonical lattice isometry  $\bs{e'} \circ \bs e^{-1} \colon \La \simeq \La'$. This guarantees that there is no ambiguity in abusing notation and speaking of \emph{the} data $(\Lambda,(\cdot,\cdot))$ associated to a quiver.

The pair $\hr{\Lambda,(\cdot, \cdot)}$ determines a \emph{quantum torus algebra} $\mathcal{T}_\Lambda^q$, which is defined to be the free $\Z[q^{\pm1}]$-module spanned by $\hc{Y_\la \,|\, \la\in\La}$, with the multiplication defined by
\beq
\label{eq:Y-mult}
q^{(\lambda,\mu)}Y_\lambda Y_\mu = Y_{\lambda+\mu}.
\eeq
A basis $\hc{e_i}$ of the lattice $\La$ gives rise to a distinguished system of generators for $\mathcal{T}_\Lambda^q$, namely the elements $Y_i=Y_{e_i}$. This way we obtain a \emph{quantum cluster $\Xc$-chart}
\beq
\label{eq:qtor-presentation}
\Tc_Q^q = \Z[q^{\pm1}]\ha{Y_i^{\pm1} \,|\, i \in I} / \ha{q^{\eps_{jk}}Y_jY_k = q^{\eps_{kj}}Y_kY_j} \simeq \Tc_\La^q.
\eeq
The generators $Y_i$ are the \emph{quantum cluster $\Xc$-variables}. We note that this presentation of $\Tc_Q^q$ depends only on the quiver and not on the choice of the representative seed. 

Let $\Theta,\Theta'$ be seeds representing quivers $Q,Q'$. We say that the quiver $Q'$ is the \emph{mutation of $Q$ in direction $k\in I$} if the map 
\beq
\label{eq:mon-mut}
\mu_k \colon \Lambda \longra \Lambda', \qquad e_i \longmapsto 
\begin{cases}
-e'_k &\text{if} \; i=k, \\
e'_i + \max\{\eps_{ik},0\}e'_k &\text{if} \; i \ne k
\end{cases}
\eeq
is an isometry. It is easy to see that $Q'=\mu_k(Q)$ if and only if $ Q=\mu_k(Q')$. The \emph{mutation class} of a quiver $Q$, which we denote by the bold symbol $\bs Q$, is the set of all quivers that can be obtained from $Q$ by a finite sequence of mutations. 

To each quiver mutation $\mu_k$ we associate an isomorphism of quantum tori
$$
\mu'_k \colon \Tc_Q^q \longra \Tc_{\mu_k(Q)}^q,
$$
and define the \emph{quantum cluster $\Xc$-mutation}
\beq
\label{eq:q-mutation}
\mu^q_k \colon \Frac(\Tc_{Q}^q) \longra  \Frac(\Tc_{Q'}^q), \qquad f \longmapsto \Psi_q\hr{Y_k'} \mu'_k(f) \Psi_q\hr{Y_k'}^{-1}
\eeq
where $\Frac(\Tc_Q)$ denotes the skew fraction field of the Ore domain $\Tc_Q$, and
$$
\Psi_q(Y) = \frac{1}{(-qY;q^2)_{\infty}} \in \mathbb{Q}(q)[[X]],
$$
is the quantum dilogarithm function. The fact that conjugation by $\Psi_q\hr{Y'_{k}}$ yields a birational automorphism is guaranteed by the integrality of the form~$(\cdot, \cdot)$ and the functional equation
$$
\Psi_q(qY) = (1+Y)\Psi_q(q^{-1}Y).
$$

\begin{defn}
An element of $\Tc^q_Q$ is said to be \emph{universally Laurent} if its image under any finite sequence of quantum cluster mutations is contained in the corresponding quantum torus algebra. The \emph{universally Laurent algebra} $\Lbb^q_{\bs Q}$ is the algebra of universally Laurent elements of $\Tc^q_Q$.
\end{defn}

The collection of quantum charts $\Tc^q_Q$ with $Q \in \bs Q$, together with quantum cluster $\Xc$-mutations is often referred to as the \emph{quantum cluster $\Xc$-variety}. We regard the quantum charts as the quantized algebras of functions on the toric charts in the atlas for the classical cluster Poisson variety. The quantum charts form an $I$-regular tree, and the cluster mutations quantize the gluing data between adjacent charts. The universally Laurent algebra is the quantum analog of the algebra of global functions on the cluster variety. Unless otherwise specified in what follows, we will simply write ``cluster variety'' for quantum cluster $\Xc$-variety --- the same applies to variables, charts, mutations, etc.

The \emph{cluster modular groupoid} associated to a cluster variety is defined as follows. 

\begin{defn}
Let $Q,Q'$ be two quivers with identical label sets $I$.  We define a \emph{permutation morphism} to be a monomial isomorphism of quantum tori $\sigma \colon \Tc_{Q}\rightarrow \Tc_{Q'}$ such that  $\sigma(Y_i)  = Y'_{\sigma(i)}$ for some permutation $\sigma$ of the set $I$. 
\end{defn}

\begin{defn}
Let $Q,Q'$ be two quivers with corresponding quantum tori $\Tc_Q,\Tc_Q'$ as in~\eqref{eq:qtor-presentation}. A \emph{cluster transformation} with source $Q$ and target $Q'$ is a non-commutative birational isomorphism $\Tc_Q\dashrightarrow \Tc_Q'$  which can be factored as a composition of cluster mutations and permutation morphisms.
\end{defn}

\begin{defn}
The \emph{cluster modular groupoid} is the groupoid $\Gc_{\bs Q}$ whose objects are quivers $Q \in \bs Q$, and whose morphisms are cluster transformations. The \emph{cluster modular group}, denoted $\Gamma_{\bs Q}$, is the automorphism group of an object in $\Gc_{\bs Q}$.
\end{defn}
\begin{remark}
Any element of the quasi-cluster modular group restricts to an automorphism of the universally Laurent algebra $\Lbb_{\bs Q}$.
\end{remark}

\subsection{Covers and $\mathcal{A}$-variables}
Suppose that $\Theta$ is a cluster seed with lattice $\Lambda$ and skew form $(\cdot,\cdot)$. We write $\Lambda_{\mathbb{Q}}$ for the vector space $\Lambda\otimes_{\mathbb{Z}}\mathbb{Q}$.
Let us denote by $\Lambda^\vee\subset\Lambda_{\mathbb{Q}}$ the abelian group
$$
\Lambda^\vee = \hc{\tilde\lambda \in \Lambda_\Qbb \,\Big|\, (\mu,\tilde\lambda)\in\mathbb{Z} \quad \forall \mu\in\Lambda}.
$$
It is a lattice if and only if $\det\eps\neq0$. Suppose that $\widetilde\Lambda$ is a lattice such that
$$
\Lambda \subseteq \widetilde\Lambda \subseteq \Lambda^\vee,
$$
and write $D$ for the smallest natural number such that
$$
\Lambda \subset \widetilde \Lambda\subset \frac{1}{D}\Lambda \subset \Lambda_{\mathbb{Q}}.
$$

Fix the primitive  $D$-th root of unity $\zeta_{D}=e^{2\pi i/D}$. We consider the quantum torus algebra $\mathcal{T}_{\widetilde\Lambda}$, which is defined to be the free $\Z[\zeta_{D},q^{\pm{\frac{1}{D}}}]$-module spanned by $\big\{Y_\la \,|\, \la\in\widetilde\La\big\}$, with the multiplication defined by~\eqref{eq:Y-mult}. Here we regard $q^{\frac{1}{D}}$ as a formal indeterminate satisfying $(q^{\frac{1}{D}})^{D}=q$. Since $\widetilde\Lambda \subseteq \Lambda^\vee$, for all $k\in I$ and $\tilde\lambda\in\widetilde\Lambda$ we have that 
$$
(\tilde\lambda,e_k)\in\mathbb{Z}.
$$
Hence the quantum mutation maps~\eqref{eq:q-mutation} extend to well-defined non-commutative birational isomorphisms
$$
\mu^q_k \colon \Frac({\Tc}_{\widetilde\Lambda}^q) \longra  \Frac({\Tc}_{\mu_k(\widetilde\Lambda)}^q),
$$
and we can therefore define an analog of the universally Laurent ring $\Lbb^q({\widetilde\Lambda})\subseteq\mathcal{T}^q_{\widetilde\Lambda}$.

In the case $\det(\eps) \ne 0$, we may take $\widetilde \Lambda=\Lambda^\vee$, which corresponds to the lattice generated by the columns of the $\mathbb{Q}$-matrix $\eps^{-1}$. Let $\{e^\vee_k\}\subset\Lambda^\vee$ be the dual basis to the basis $\{e_k\}$ of $\Lambda$, in the sense that
$$
(e_i,e^\vee_j)=\delta_{ij}.
$$
The elements
$$
Y_{e^{\vee}_k} \in \mathcal{T}_{\Lambda^\vee}, \quad k\in I
$$
are called the \emph{quantum cluster $\mathcal{A}$-variables.} By the quantum Laurent phenomenon~\cite{BZ05}, the quantum $\mathcal{A}$-variables from each cluster are elements of the covering universally Laurent ring
$$
\widehat{\Lbb}^q:=\Lbb^q(\Lambda^\vee).
$$

\section{Quantum cluster varieties from moduli spaces of framed local systems}
In this section, we consider several examples of quantum cluster varieties coming from moduli space of framed $PGL_2$ or $SL_2$ local systems on surfaces with punctures and marked points.

\subsection{Cylinder} 
\label{sec:cylinder}
Let $C$ be the cylinder with one marked point on each boundary component. The moduli space $\mathcal{X}_{C,PGL_2}$ of framed $PGL_2$-local systems on $C$ is cluster Poisson, and its cluster modular groupoid has two objects corresponding to the quivers shown in Figure~\ref{fig:cylinder}. In this case we have $\det\eps=4$, and the dual basis to $e_1,e_2$ is given by
$$
e_1^\vee =-\frac{1}{2}e_2, \quad e_2^\vee =\frac{1}{2}e_1.
$$
We can thus define the covering universally Laurent ring $\widehat \Lbb^q_{cyl}$ associated to the lattice $\Lambda^\vee$. Since $D=2$, $\widehat \Lbb^q_{cyl}$ is an algebra over $\mathbb{Z}[q^{\pm\frac{1}{2}}]$.

\begin{figure}[h]
\begin{tikzpicture}[every node/.style={inner sep=0, minimum size=0.45cm, thick, draw, circle}, thick, x=0.8cm, y=0.8cm]

\node[fill=white] (1) at (-1,0) {\footnotesize{1}};
\node[fill=white] (2) at (1,0) {\footnotesize{2}};

\draw[->] (2.165) to (1.15);
\draw[->] (2.-165) to (1.-15);

\begin{scope}[shift={(5,0)}]
\node[fill=white] (1) at (-1,0) {\footnotesize{1}};
\node[fill=white] (2) at (1,0) {\footnotesize{2}};

\draw[<-] (2.165) to (1.15);
\draw[<-] (2.-165) to (1.-15);
\end{scope}

\end{tikzpicture}
\caption{Quivers $Q_{cyl}$ and $Q'_{cyl}$.}
\label{fig:cylinder}
\end{figure}
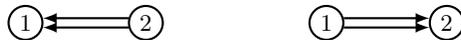

The corresponding cluster algebra has $\mathbb{Z}$-many clusters. We fix a basepoint in this tree given by an \emph{initial cluster} living over the quiver $Q_{cyl}$. The covering universally Laurent ring $\widehat \Lbb^q_{cyl}$ contains the quantum $\Ac$-variables 
$$
Y_{e^\vee_1}=Y_{-\frac12 e_2},\quad Y_{e^\vee_2}= Y_{\frac12 e_1},
$$ 
from this initial cluster, as well as the trace $L$ of the monodromy around the cylinder. The latter can be expressed in cluster coordinates as
\beq
\label{eq:trace-cyl}
L = Y_{-\frac12(e_1+e_2)} + Y_{\frac12(e_2-e_1)} + Y_{\frac12(e_1+e_2)},
\eeq
and by the $GL_n$ case considered in~\cite{SS19}, we have $L \in \widehat \Lbb^q_{cyl}$. 

Let us briefly recall the standard combinatorial recipe used to obtain formula~\eqref{eq:trace-cyl}, see e.g.~\cite[Section 9]{FG06}. Consider the bipartite graph on a cylinder shown in the left part of Figure~\ref{fig:network}. On the right we see the dual quiver with edges directed in such a way that the white vertex of the bipartite graph is on the right as we traverse an edge. Note that upon identifying the pair of nodes with label 2, we recover the quiver $Q_{cyl}$. The direction of edges of the bi-partite graph is additional data, which allows one to express a monodromy matrix $M$ in cluster coordinates. Namely, we set
$$
M_{ij} = \sum_{p : \, j \to i} Y_{\wt(p)}, \qquad\text{where}\qquad \wt(p) = \sum_{f \, \text{below} \, p} e_f.
$$
The first sum in the formula above is taken over all paths $p$ going from $i$-th source to the $j$-th sink in the directed bipartite graph, while the second is taken over all faces lying below the path $p$. Since the monodromy matrix is defined up to conjugation, we shall only consider its trace $L$. Finally, setting $y_0 = -\frac12(y_1+y_2)$, we obtain $\det(L)=1$, and recover formula~\eqref{eq:trace-cyl}.

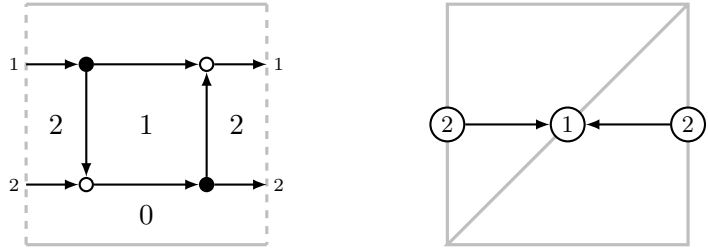
\begin{figure}[h]
\begin{tikzpicture}[every node/.style={inner sep=0, minimum size=0.45cm, thick, draw, circle}, thick, x=0.8cm, y=0.8cm]

\begin{scope}[shift={(7,0)}]
\draw[very thick, gray!50] (-2,-2) to (-2,2) to (2,2) to (2,-2) to (-2,-2) to (2,2);

\node[fill=white] (1) at (0,0) {\footnotesize{1}};
\node[fill=white] (2_1) at (-2,0) {\footnotesize{2}};
\node[fill=white] (2_2) at (2,0) {\footnotesize{2}};

\draw [->] (2_1) -- (1);
\draw [->] (2_2) -- (1);
\end{scope}

\draw[very thick, dashed, gray!50] (-2,-2) to (-2,2);
\draw[very thick, gray!50] (-2,2) to (2,2);
\draw[very thick, dashed, gray!50] (2,2) to (2,-2);
\draw[very thick, gray!50] (2,-2) to (-2,-2);

\node[fill=white, draw=none] at (0,0) {1};
\node[fill=white, draw=none] at (-1.5,0) {2};
\node[fill=white, draw=none] at (1.5,0) {2};
\node[fill=white, draw=none] at (0,-1.5) {0};

\node[fill=white, minimum size=5pt] (p1p1) at (1,1) {};
\node[fill=white, minimum size=5pt] (m1m1) at (-1,-1) {};
\node[fill, minimum size=5pt] (m1p1) at (-1,1) {};
\node[fill, minimum size=5pt] (p1m1) at (1,-1) {};

\draw[->] (-2,1) to (m1p1);
\draw[->] (m1p1) to (p1p1);
\draw[->] (p1p1) to (2,1);
\draw[->] (-2,-1) to (m1m1);
\draw[->] (m1m1) to (p1m1);
\draw[->] (p1m1) to (2,-1);
\draw[->] (m1p1) to (m1m1);
\draw[->] (p1m1) to (p1p1);

\node[fill=white, draw=none, minimum size=0] at (-2.2,1) {\tiny 1};
\node[fill=white, draw=none, minimum size=0] at (-2.2,-1) {\tiny 2};
\node[fill=white, draw=none, minimum size=0] at (2.2,1) {\tiny 1};
\node[fill=white, draw=none, minimum size=0] at (2.2,-1) {\tiny 2};

\end{tikzpicture}
\caption{Directed network and dual cluster quiver.}
\label{fig:network}
\end{figure}

The mapping class group of the cylinder is isomorphic to $\Z$ and is generated by a signle Dehn twist, which we denote $\tau$. It can be realized as a quantum cluster transformation
$$
\tau = (1\,2) \circ \mu_1^q.
$$
The latter is a quantization of the \emph{discrete Toda flow}, see~\cite{HKKR00, Wil15}, also known as the \emph{quantum $Q$-system}, see~\cite{DFK10, DFK11}.\footnote{In~\cite{DFK10, DFK11}, the $q$-Whittaker limit of Macdonald operators is taken at $t \to \infty$ rather than at $t \to 0$, which leads to the discrete time evolution $\tau$ in the present text being inverse of that in \emph{loc.cit.}} Let us put
$$
A_n = \tau^{-n}(Y_{e^{\vee}_2}),
$$
so that we have
$$
A_0 = Y_{e^{\vee}_2}, \quad A_{-1} = Y_{e^{\vee}_1}.
$$
The elements $\{A_n\}_{n\in\mathbb{Z}}$ form the set of all quantum $\mathcal{A}$-variables, and any two adjacent ones $(A_n,A_{n+1})$ form a cluster. Since the Kronecker quiver is acyclic and $\det(\eps) \ne 0$, they generate the universally Laurent algebra $\widehat \Lbb^q_{cyl}$, see~\cite{BZ05}. The element $L$ is invariant under the Dehn twist: $\tau(L) = L$. Indeed, it may be viewed as the ``infinitesimal generator'' of the $Q$-system evolution in the sense that 
\beq
\label{eq:infinitesimal-dehn}
q^{\pm\frac{1}{2}}A_nL - q^{\mp\frac{1}{2}}LA_n = (q^{\pm1}-q^{\mp1})A_{n\pm1}.
\eeq
Thus, the elements $L$ and $A_0$ generate the universally Laurent algebra $\widehat\Lbb^q_{cyl}$ over over the ring
$$
\Z\Big[q^{\pm\frac12},(q-q^{-1})^{-1}\Big].
$$

As discussed in~\cite{DFK18}, the formulas
\beq
\label{eq:cyl-rep}
A_n \longmapsto \ibf q^{-\frac{1}{2}}\check H_n, \qquad L \longmapsto x+x^{-1}
\eeq
define a representation of the algebra $\widehat \Lbb^q_{cyl}$ on the ring $\mathcal{S}_{q,t}$. Let us recover this representation. We start by considering the ring $\Vc$ of compactly supported $\mathbb{C}[q^{\pm1}]$-valued functions on the lattice $\Z$. The space $\Vc$ carries an action of the quantum torus
$$
\Dc_q = \Z[q^{\pm1}]\ha{U,V}/\ha{UV=qVU}
$$
defined by formulas
\beq
\label{eq:repdef}
(U f)(n) = q^n f(n) \qquad\text{and}\qquad (V f)(n) = f(n-1).
\eeq
We embed the covering cluster torus $\mathcal{T}^q({\Lambda^\vee_{Q_{cyl}}})$ into $\Dc_q$ by 
$$
Y_{\frac12e_1} \longmapsto \ibf q^{-\frac{1}{2}}U^{-1}, \qquad Y_{\frac12e_2} \longmapsto -\ibf V^{-1}U.
$$

Although the following Lemma is well known, we include a proof for the reader's convenience.

\begin{lemma}
The representation $\Vc$ of $\Tc^q$ defined by~\eqref{eq:repdef} is faithful.
\end{lemma}

\begin{proof}
A general element of the algebra $\Tc^q$ can be written as
$$
a = \sum_{n,m = -N}^N a_{n,m}V^mU^n
$$
for some integer $N \in \Z$ and coefficients $a_{n,m}\in\mathbb{Z}[q^{\pm1}]$. If $a$ acts by zero in $\Vc$ then in particular it annihilates each indicator function $\{\delta_l\}_{l\in\mathbb{Z}}$ defined by
$$
\delta_l(n) =
\begin{cases} 1 &\text{if} \quad n=l, \\
0& \text{otherwise}.
\end{cases}
$$
Since $U\delta_l = q^l\delta_{l}$ and $V\delta_l = \delta_{l+1}$, we see that
$$
a\cdot\delta_l = \sum_{n,m = -N}^N a_{n,m}q^{n(l+m)}\delta_{l+m}.
$$
So if $a$ acts by zero in $\mathcal{H}$, then for all $k,l\in\mathbb{Z}$ we have
\beq
\label{eq:syseq}
\sum_{n=-N}^Na_{n,k}q^{n(l+k)}=0
\eeq
Now for each $k$ we may regard~\eqref{eq:syseq} as a system of infinitely many equations, one for each $l\in\Z$, in the $2N+1$ variables $a_{n,k}$. Take $2N+1$ of them given by letting $l$ run from $-k$ to $2N-k$. Then the coefficient matrix of the resulting system is
$$
C = \left(q^{(N-i)j}\right)_{0\le i,j\le 2N},
$$
and we have
$$
\det(C) = \prod_{0\le i<j\le 2N}(q^{N-i}-q^{N-j}).
$$
In any ring where $q$ is transcendental, in particular in $\Z[q^{\pm1}]$, this determinant is nonzero, and hence $a_{n,k}=0$ for all $n$. 
This shows that $a=0$, so the representation is faithful. 
\end{proof}

In particular, we get a a faithful representation of the $Q$-system cluster algebra $\widehat \Lbb^q_{cyl}$
$$
\rho_0 \colon \widehat \Lbb^q_{cyl}\longra \mathrm{End}(\mathcal{V}),
$$
whose generators $L,A_0$ act by
\beq
\label{eq:torus-embed-gens}
A_0 \longmapsto \ibf q^{-\frac{1}{2}}U^{-1}, \qquad L \longmapsto V + V^{-1}\hr{1-U^2}.
\eeq
Now let $\Fc \subset \Vc$ be the subring of functions with support in $\Z_{\ge0}$. Note that the action of the operator $V^{-1}$ does not preserve the subspace $\Fc$, so it is not a module over the entire quantum torus $\Dc_q$. Rather, an element $a\in\Dc_q$ gives rise to a linear map $a \colon \Fc \to \Vc$. The same standard argument used to establish faithfulness of the representation $\Vc$ shows that no nonzero element of $\Dc_q$ can annihilate the entire subspace $\Fc$. On the other hand, we observe that the generators $A_0,L$ for $\widehat \Lbb^q_{cyl}$ do in fact preserve the subspace $\Fc$, and $\mathcal{F}$ is therefore an $\widehat \Lbb^q_{cyl}$-submodule in the representation $\mathcal{V}$.

Now we use the Whittaker basis from Section~\ref{sec:genus-0} to identify the vector space $\Fc$ with the ring $\mathcal{S}_{q,t}$:
\beq
\label{eq:whittaker-iso}
\bs W \colon \Fc \longra \mathcal{S}_{q,t}, \qquad \phi \longmapsto \sum_{l\ge 0}\phi(l)W_{l}(x).
\eeq
This equips $\mathcal{S}_{q,t}$ with an $\widehat \Lbb^q_{cyl}$ action via~\eqref{eq:torus-embed-gens}. Comparing these formulas for the action of $A_0,L$ with the $\check H_0$ eigenproperty~\eqref{eq:Whit-dual-eigen-prop} and the Pieri rule~\eqref{eq:Whit-Pieri}, we finally see that the generators $A_0,L$ of $\widehat \Lbb^q_{cyl}$ act on $\mathcal{S}_{q,t}$ via formulas~\eqref{eq:cyl-rep}.

Note that we can use the faithfulness of the two representations to reverse the logic: recalling the algebra $\mathbb{SH}_{g=0}$ of $q$-difference operators in $x$ generated by all $\check H_n$ and $L$, and writing $\Dc_q^{\mathcal{F}}$ for the subalgebra of $\Dc$ preserving $\Fc\subset\Vc,$ we get an injective algebra homomorphism
\beq
\label{eq:iota0}
\eta_0 \colon \mathbb{SH}_{g=0}\longra \Dc^{\mathcal{F}}_q \subset \Dc_q
\eeq
obtained  by identifying both sides with subalgebras of $\mathrm{End}_{\mathbb{C}(q)}(\Fc)$. Furthermore, since the image $\eta_0\hr{\SH_{g=0}} \subset \Dc^{\mathcal{F}}_q$ is contained inside the image $\rho_0\big(\widehat\Lbb^q_{cyl}\big) \subset \Dc^{\mathcal{F}}_q$, and the actions of the mapping class group of the cylinder on both algebras are compatible, we arrive at the following well-known result.

\begin{prop}
There exists a $\Z$-equivariant isomomorphism
$$
\iota \colon \SH_{g=0} \longra \widehat \Lbb^q_{cyl}
$$
defined by inverting the formula~\eqref{eq:cyl-rep}.
\end{prop}

For each of the $\mathbb{Z}$-many clusters in the atlas, the lattice isomorphism~\eqref{eq:mon-mut} together with formula~\eqref{eq:torus-embed-gens} determines an embedding of the corresponding quantum torus into $\Dc_q$. Restriction to the universally Laurent ring $\widehat\Lbb_{cyl}$ then defines a new representation of the latter on the space $\mathcal{V}$, which may not be equivalent to the original.

For example, consider the cluster obtained by mutating the initial one at vertex 1. Then the corresponding embedding is
$$
\begin{aligned}
Y_{\frac{1}{2}e_1'} &= Y_{-\frac{1}{2}e_1} & & \hspace{-1em}\longmapsto -\mathbf{i}q^{\frac{1}{2}}U, \\
Y_{\frac{1}{2}e_2'} &= Y_{\frac{1}{2}e_2+e_1} & & \hspace{-1em}\longmapsto \mathbf{i}V^{-1}U^{-1}.
\end{aligned}
$$
The automorphism part of the mutation is the conjugation by $\Psi(Y_{-e_1})$, so that we have
$$
\mu_{1}^q(L) = Y_{-\frac{1}{2}(e'_1+e'_2)}+ Y_{\frac{1}{2}(e'_1-e'_2)}+ Y_{\frac{1}{2}(e'_1+e'_2)}.
$$
Hence in the new representation 
$$
\rho_{-1} \colon \widehat \Lbb^q_{cyl}\longra \mathrm{End}(\mathcal{V}),
$$ the element $L\in \widehat\Lbb_{cyl}$ acts by
\beq
\label{eq:rho-L}
\rho_{-1}(L) = V^{-1} + (1-U^2)V,
\eeq
while we still have
$$
\rho_{-1}(A_0) = \ibf q^{-\frac{1}{2}}U^{-1}.
$$
Since $Y_{e_1'}\mapsto -qU^2$, the mutated counterpart of the embedding~\eqref{eq:iota0} is
\beq
\label{eq:iota1}
\eta_{-1}\colon \SH_{g=0} \longra \Dc_q
\qquad\text{where}\qquad
\eta_{-1} = \Ad_{\Psi_q(-qU^2)}\circ \eta_0.
\eeq
Note that the image of $\eta_{-1}$ is no longer contained in $\Dc_q^\Fc$. For example, we have
\beq
\label{eq:eta-H-check}
\begin{aligned}
\check H_0 &\longmapsto U^{-1}, \\
\check H_1 &\longmapsto q^{\frac12} V^{-1}U^{-1}, \\[2pt]
\check H_2 &\longmapsto q\hr{V^{-2}U^{-1}-U}.
\end{aligned}
\eeq
As noted above, in the new representation $\rho_{-1}$ the algebra $\widehat\Lbb^q_{cyl}$ no longer preserves the subspace $\mathcal{F}\subset\mathcal{V}$, but instead preserves the ideal $\Vc_+ \subset \Vc$ of functions vanishing on $\Z_{\ge0}$. The representations $\Vc_+\subset(\Vc,\rho_{-1})$ and $\mathcal{F}\subset(\Vc,\rho_0)$ of $\widehat\Lbb^q_{cyl}$ are non-isomorphic, as can be seen from the corresponding sets of eigenvalues of the element $A_0$.

Restriction of functions to $\Z_{\ge0}$ defines a short exact sequence of $\widehat\Lbb^q_{cyl}$-modules
$$
0 \longra \Vc_+ \longra (\Vc,\rho_{-1}) \longra \Fc \longra 0.
$$
The two $\widehat\Lbb^q_{cyl}$-module structures on $\mathcal{F}$, one coming as the kernel of $\rho_0$ and the other as the cokernel of $\rho_{-1}$, are isomorphic. Indeed, consider the distribution $\Psi_q^+[n]$ defined by
$$
\Psi_q^+[n]=
\begin{cases}
1/(q^2;q^2)_n, & n\ge0 \\
0, & n<0.
\end{cases}
$$
Since $\Psi_q^+[n]$ satisfies the difference equation 
$$
\Psi_q^+[n-1] = (1-q^{2n})\Psi_q^+[n],
$$
we observe that the multiplication operator
\beq
\label{eq:intertwining}
\mu_{1} \colon (\Vc,\rho_{-1}) \longrightarrow (\Vc,\rho_{0}), \qquad f(n) \longmapsto \Psi_q^+[n] f(n)
\eeq
intertwines the indicated representations of $\widehat\Lbb^q_{cyl}$. Moreover, since $\Psi_q^+$ vanishes on all negative integers, we see that $\mu_{1}$ descends to an isomorphism 
\beq
\label{eq:discrete-mut}
\mu_{1} \colon (\Vc,\rho_{-1})/\Vc_+ \longrightarrow  \Fc \subset (\Vc,\rho_0).
\eeq
On the other hand, the map 
$$
(\Vc,\rho_{0}) \longra (\Vc_+,\rho_{-1}), \qquad f(n) \longmapsto \Psi^{+}_{q^{-1}}[n-1]f(n)
$$
intertwines the $\widehat\Lbb^q_{cyl}$ actions, and its kernel is precisely the submodule $\mathcal{F}\subset(\Vc,\rho_{0}) $. Thus the mutation $\mu_{1}^q$ manifests itself via a pair of short exact sequences of $\widehat\Lbb^q_{cyl}$-modules:
$$
\begin{tikzcd}
0 \arrow[r] & \Vc_+ \arrow[r] & \Vc \arrow[r,"i^*"]  & \Fc \arrow[r] & 0
\end{tikzcd}
$$
and
$$
\begin{tikzcd}
0 \arrow[r] & \Fc \arrow[r, "i_!"] & \Vc \arrow[r] & \Vc_+ \arrow[r] & 0
\end{tikzcd}
$$
where the algebra $\widehat\Lbb^q_{cyl}$ acts via $\rho_{-1}$ on $\Vc$ in the top sequence, and via $\rho_0$ in the bottom one.

\subsection{Punctured torus.}

We now recall the cluster structure on the moduli space of framed $SL_2$ local systems on the punctured torus. A detailed discussion of the $GL_2$ case can be found in~\cite{DFK+24}, and we refer the reader to \emph{loc. \!\!cit.\!} and references therein for further details.

Consider the Markov quiver, see Figure~\ref{fig:Markov}. The corresponding skew-form is degenerate, and its kernel is spanned by the vector $z = e_1+e_2+e_3$. We will work with the lattice
$$
\widetilde\Lambda = \frac{1}{2}\Lambda \subset \Lambda^\vee,
$$
and write $\widehat \Lbb^q_{tor}$ for the corresponding universally Laurent algebra.

\begin{figure}[h]
\begin{tikzpicture}[every node/.style={inner sep=0, minimum size=0.45cm, thick, draw, circle}, thick, x=1cm, y=0.866cm]

\node (1) at (0,0) {\footnotesize{1}};
\node (2) at (-1,2) {\footnotesize{2}};
\node (3) at (1,2) {\footnotesize{3}};

\draw [->] (3.165) -- (2.15);
\draw [->] (1.45) -- (3.-105);
\draw [->] (1.75) -- (3.-135);
\draw [->] (2.-45) -- (1.105);
\draw [->] (2.-75) -- (1.135);
\draw [->] (3.-165) to (2.-15);

\end{tikzpicture}
\caption{Markov quiver $Q$.}
\label{fig:Markov}
\end{figure}
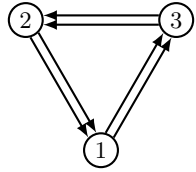

For $(m,n)$ coprime, denote by $L_{(m,n)}$ the quantum trace of the holonomy along the $(m,n)$-curve on the torus. Note that unlike in the $GL_2$ case, here we have $L_{(a,b)} = L_{(-a,-b)}$. Let us choose a basis in the lattice $H_1(T^2\setminus D^2;\Z) \simeq \mathbb{Z}^2$ in such a way that
\beq
\label{eq:L-torus}
\begin{aligned}
L_{(1,0)} &= Y_{-\frac12(e_1+e_2)} + Y_{\frac12(e_2-e_1)} + Y_{\frac12(e_1+e_2)}, \\
L_{(0,1)} &= Y_{-\frac12(e_1+e_3)} + Y_{\frac12(e_1-e_3)} + Y_{\frac12(e_1+e_2)}.
\end{aligned}
\eeq

The mapping class group of a punctured torus is isomorphic to $SL(2,\Z)$, and is generated by elements
$$
\sigma_+ =
\begin{pmatrix}
1 & 1 \\
0 & 1
\end{pmatrix}
\qquad\text{and}\qquad
\sigma_- =
\begin{pmatrix}
1 & 0 \\
1 & 1
\end{pmatrix}
$$
which correspond to the Dehn twists of the torus along simple closed curves with homology classes $(1,0)$ and $(0,1)$ respectively. By the construction in~\cite[Section 6]{FG09}, we get a homomorphism $SL(2,\Z) \to \Gamma_{\bs Q}$ defined by
$$
\tau_+^{-1} \longmapsto (1 \, 2) \circ \mu_1^q, \qquad
\tau_- \longmapsto (1\,3) \circ \mu_3^q.
$$
The element $\tau$ of order 6 defined by
$$
\tau = \tau_+^{-1}\tau_- = \begin{pmatrix} 0 & -1 \\ 1 & 1 \end{pmatrix}
$$
is mapped under this homomorphism to a permutation $(3\;2\;1) \in S_3$, and thus the homomorphism $SL(2,\Z) \to \Gamma_{\bs Q}$ factors through $PSL(2,\Z)$. Given an element $g \in PSL(2,\Z)$ we set $L_{g \cdot v} = g \cdot L_v$ for $v \in \Z^2$. This definition makes sense due to the fact that $\tau_+$ preserves $L_{(1,0)}$. It is also easy to check that the definition at hand is compatible with formulas~\eqref{eq:L-torus}. As before, the element $L_{(1,0)}$ lies in the corresponding universally Laurent algebra $\widehat \Lbb^q_{tor}$, and hence so does the element $L_v$ for any primitive vector $v \in \Z^2$.

Consider a quantum torus
$$
\Dc_q[t^{\pm1}] = \Dc_q \otimes_{\Z[q^{\pm1}]} \Z[q^{\pm1},t^{\pm1}].
$$
Then, similarly to the genus 0 case, we obtain an injective homomorphism
$$
\Tc_Q^q \longra \Dc_q[t^{\pm1}],
$$
given by the formulas
$$
Y_{\frac12e_1} \longmapsto \ibf V^{-1}U^{-1}, \qquad
Y_{\frac12e_2} \longmapsto -\ibf q^{\frac12} U, \qquad
Y_{\frac12e_3} \longmapsto -\ibf q^{\frac12} t^{-1}V.
$$
Note that we have
$$
Y_{e_1+e_2+e_3} \longmapsto -qt^{-2},
$$
as well as
\begin{align*}
L_{(1,0)} &\longmapsto V^{-1} + \hr{1-U^2}V, \\
L_{(0,1)} &\longmapsto t^{-1}U^{-1}+ t\hr{U - V^{-2}U^{-1}}.
\end{align*}

Recalling formulas~\eqref{eq:rho-L} and~\eqref{eq:eta-H-check}, we see that
\begin{align*}
V^{-1} + \hr{1-U^2}V &= \eta_{-1}(x+x^{-1}), \\
t^{-1}U^{-1}+ t\hr{U - V^{-2}U^{-1}} &= \eta_{-1}\hr{t^{-1}\check H_0 - q^{-1}t \check H_2}.
\end{align*}
In view of the expression~\eqref{eq:Macdo-dual-Toda} of Macdonald operator in terms of the operators $\check H_n$ and the description of the $SL(2,\Z)$ action on $ \SH_{g=1}$ and $\widehat\Lbb^q_{tor}$, we arrive at the following result. We also refer the reader to~\cite{DFK+24} for more details and the $GL_2$ version of it.

\begin{prop}
There exists an $SL(2,\Z)$-equivariant injective homomorphism
$$
\iota \colon \SH_{g=1} \longrightarrow \widehat\Lbb^q_{tor},
$$
such that
$$
x+x^{-1} \longmapsto L_{(1,0)} \qquad\text{and}\qquad M \longmapsto L_{(0,1)}.
$$
\end{prop}

\subsection{Closed surface of genus 2.}
\label{sec:g2thmsec}
As was shown in~\cite{CS23}, the quiver $X_7$, see Figure~\ref{fig:X7}, describes a cluster structure on a 1-parameter deformation of the ring of functions on the Teichm\"uller space for $\Sigma_{2,0}$. We denote the corresponding universally Laurent algebra by $\widehat \Lbb^q_{g=2}$. In analogy with the genus 1 case, the universally Laurent ring contains elements corresponding to the  traces of holonomies around the loops $A_j$ and $B_{ij}$. The latter are written in cluster coordinates as
\beq
\label{eq:B-cycles}
\begin{aligned}
L_{B_{12}} &= Y_{-\frac12(e_{5}+e_{6})} + Y_{\frac12(e_{6}-e_{5})} + Y_{\frac12(e_{5}+e_{6})}\\
L_{B_{13}} &= Y_{-\frac12(e_{3}+e_{4})} + Y_{\frac12(e_{4}-e_{3})} + Y_{\frac12(e_{3}+e_{4})}\\
L_{B_{23}} &= Y_{-\frac12(e_{1}+e_{2})} + Y_{\frac12(e_{2}-e_{1})} + Y_{\frac12(e_{1}+e_{2})}.
\end{aligned}
\eeq
Formulas for the former are more cumbersome, and are best described using the cluster realization of the mapping class group.

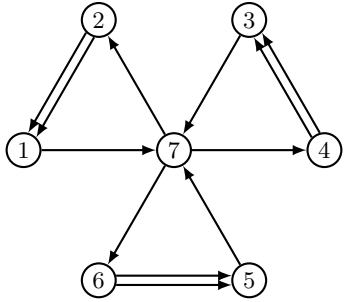
\begin{figure}[h]
\begin{tikzpicture}[every node/.style={inner sep=0, minimum size=0.45cm, thick, draw, circle}, thick, x=1cm, y=0.866cm]

\node (1) at (-2,0) {\footnotesize{1}};
\node (2) at (-1,2) {\footnotesize{2}};
\node (3) at (1,2) {\footnotesize{3}};
\node (4) at (2,0) {\footnotesize{4}};
\node (5) at (1,-2) {\footnotesize{5}};
\node (6) at (-1,-2) {\footnotesize{6}};
\node (7) at (0,0) {\footnotesize{7}};

\draw [->] (7) to (2);
\draw [->] (2.-105) to (1.45);
\draw [->] (2.-135) to (1.75);
\draw [->] (1) to (7);
\draw [->] (7) to (4);
\draw [->] (4.105) to (3.-45);
\draw [->] (4.135) to (3.-75);
\draw [->] (3) to (7);
\draw [->] (7) to (6);
\draw [->] (6.15) to (5.165);
\draw [->] (6.-15) to (5.-165);
\draw [->] (5) to (7);

\end{tikzpicture}
\caption{The quiver $X_7$.}
\label{fig:X7}
\end{figure}

The mapping class group of $\Sigma_{2,0}$ is generated by the Dehn twists along the curves $(A_1,A_2,A_3)$ and $(B_{12},B_{13},B_{23})$. As before, the $B$-cycle Dehn twists are given by
$$
\tau_{B_{12}} = (5\,6) \circ \mu_5^q, \qquad 
\tau_{B_{13}} = (3\,4) \circ \mu_3^q, \qquad 
\tau_{B_{23}} = (1\,2) \circ \mu_1^q.
$$
In~\cite{CS23}, the authors considered the cluster modular group element
\beq
\label{eq:gamma-elt}
\gamma = (1\,2)(3\,4)(5\,6) \circ \mu^q_7 \mu^q_5 \mu^q_3 \mu^q_1 \mu^q_7.
\eeq
It can be computed that for the the semi-classical limits $G_{A_{ij}}$ of $L_{A_{ij}}$ and $G_{B_k}$ of $L_{B_k}$, see section~\ref{sec:class}, one has
$$
\gamma(G_{A_k}) = G_{B_{ij}},
$$
for any permutation $(i,j,k)$  of $(1,2,3)$ with $i<j$. Thus, we define elements $L_{A_k} \in \widehat \Lbb^q_{\bs{X_7}}$ by
$$
L_{A_k} = \gamma^{-1}(L_{B_{ij}}),
$$
and arrive at formulas:
\beq
\label{eq:A-cycles}
\begin{aligned}
L_{A_1} = Y_{e_7+\frac12(e_2+e_1+e_3+e_5)} + L_{B_{23}}Y_{\frac12(e_3+e_5)} + Y_{-e_7-\frac12(e_2+e_1+e_3+e_5)}\hr{1+qY_{e_3}}\hr{1+qY_{e_5}}, \\
L_{A_2} = Y_{e_7+\frac12(e_4+e_1+e_3+e_5)} + L_{B_{13}}Y_{\frac12(e_1+e_5)} + Y_{-e_7-\frac12(e_4+e_1+e_3+e_5)}\hr{1+qY_{e_1}}\hr{1+qY_{e_5}}, \\
L_{A_3} = Y_{e_7+\frac12(e_6+e_1+e_3+e_5)} + L_{B_{12}}Y_{\frac12(e_1+e_3)} + Y_{-e_7-\frac12(e_6+e_1+e_3+e_5)}\hr{1+qY_{e_1}}\hr{1+qY_{e_3}}.
\end{aligned}
\eeq
The same argument used in the $g=0$ and $g=1$ cases shows that the elements $L_{B_{ij}}$ are universally Laurent, and hence so are the $L_{A_k}$. 

Combining formulas~\eqref{eq:infinitesimal-dehn} and~\eqref{eq:collect-coeffs}, we obtain:
\begin{lemma}
\label{eq:lem-dehn-b}
The $B$-cycle Dehn twists preserve the $B$-cycle trace functions and act on the $A$-cycle trace functions by
\beq
\tau_{B_{ij}}^{\pm1}(L_{A_k}) =  \begin{cases} \pm(q-q^{-1})^{-1}\hr{q^{\frac{1}{2}}L_{A_k}L_{B_{ij}} - q^{-\frac{1}{2}}L_{B_{ij}}L_{A_k}} \quad &k\in\{i,j\}\\
L_{A_k}\quad &k\notin\{i,j\}
\end{cases}
\eeq
\end{lemma}

Another useful modular group element is the involution 
\beq
\label{eq:sigma-elt}
\sigma= (1\,5)(3\,7)\circ \mu_7\mu_3\mu_5\mu_1\mu_7\mu_3.
\eeq
A simple calculation shows that for any permutation $(i,j,k)$  of $(1,2,3)$ with $i<j$, the involution $\sigma$ acts on the elements $L_{B_{ij}}$ and $L_{A_k}$ by
$$
\sigma(L_{B_{ij}}) = L_{A_k}, \qquad \sigma(L_{A_k}) = L_{B_{ij}}.
$$
Hence in the cluster obtained from the initial one by applying the element $\sigma$, the very same argument used to  derive the formulas in Lemma~\ref{eq:lem-dehn-b} yields:
\begin{lemma}
\label{eq:lem-dehn-a}
The $A$-cycle Dehn twists preserve the $A$-cycle trace functions and act on the $B$-cycle trace functions by
$$
\tau_{A_k}^{\pm1}(L_{B_{ij}}) =
\begin{cases}
\pm (q-q^{-1})^{-1}\hr{q^{\frac{1}{2}}L_{B_{ij}}L_{A_k} - q^{-\frac{1}{2}}L_{A_k}L_{B_{ij}}} \quad &k\in\{i,j\}\\
L_{B_{ij}}\quad &k\notin\{i,j\}
\end{cases}
$$
\end{lemma}

To compare the subalgebra of $\widehat \Lbb^q_{X_7}$ generated by the elements $L_{A_k}$ and $L_{B_{ij}}$ with the algebra $\SH_{g=2}$ defined in~\cite{AS19}, let us consider the quantum torus $\Dc_q^{\otimes 3}$ generated over $\Z(q^{\pm1},t^{\pm1})$ by elements $U_{ij},V_{ij} $ for $1\le\ i<j\le 3$ and nontrivial commutation relations
$$
U_{ij}V_{kl} = q^{\delta_{ik}\delta_{jl}}V_{kl}U_{ij}.
$$
It has a representation $\Vc^{\otimes3}$ which is identified with the ring of compactly supported functions of $(j_{12},j_{13},j_{23})\in\mathbb{Z}^3$. The assignments
$$
\begin{aligned}
X_{\frac{1}{2}e_1} &\longmapsto \mathbf{i}V_{23}^{-1}U_{23}^{-1}, &\qquad X_{\frac{1}{2}e_2} &\longmapsto -\mathbf{i}q^{\frac{1}{2}}U_{23}, \\
X_{\frac{1}{2}e_3} &\longmapsto \mathbf{i}V_{13}^{-1}U_{13}^{-1}, & X_{\frac{1}{2}e_4} &\longmapsto -\mathbf{i}q^{\frac{1}{2}}U_{13}, \\
X_{\frac{1}{2}e_5} &\longmapsto \mathbf{i}V_{12}^{-1}U_{12}^{-1}, &X_{\frac{1}{2}e_6} &\longmapsto -\mathbf{i}q^{\frac{1}{2}}U_{12}, \\
\end{aligned}
$$
and
$$
X_{e_7} \longmapsto -qt^{-1}V_{12}V_{13}V_{23}.
$$
define an injective homomorphism
$$
\underline \rho \colon \Lbb^q_{X_7}\hookrightarrow \Dc_q^{\otimes3}.
$$

On the other hand, recall from formula~\eqref{eq:collect-coeffs} that we have
$$
\mathbb{SH}_{g=2}\subset\mathbb{SH}_{g=0}^{\otimes3} \otimes_{\C(q)} \C(q,t).
$$
So we can use the map $\eta_{-1}\colon \mathbb{SH}_{g=0}\longra \Dc_q$ from formula~\eqref{eq:iota1} to define an algebra embedding
$$
\eta_{-1}^{\otimes 3} \colon \mathbb{SH}_{g=0}^{\otimes 3}\longra \Dc_q^{\otimes 3}.
$$

\begin{theorem}
\label{thm:g2-embed}
There exists a $\Gamma_{2,0}$-equivariant injective algebra homomorphism
$$
\mathbb{SH}_2 \longrightarrow \Lbb^q_{X_7},
$$
defined by
$$
\Oc_{A_k}\longmapsto L_{A_k} \qquad\text{and}\qquad \Oc_{B_{ij}} \longmapsto L_{B_{ij}}.
$$
\end{theorem}

\begin{proof}
The existence of the homomorphism follows immediately from observing that both sides have the same image under their respective embeddings to $\Dc_q^{\otimes 3}$: putting together formulas~\eqref{eq:collect-coeffs} and~\eqref{eq:eta-H-check}, we see that
$$
\eta_{-1}^{\otimes 3}(\Oc_{A_k}) = \underline\rho(L_{A_k}) \qquad\text{and}\qquad \eta_{-1}^{\otimes 3}(\Oc_{B_{ij}}) = \underline\rho(L_{B_{ij}}).
$$
The equivariance under the action of the mapping class group follows from Lemmas~\ref{eq:lem-dehn-b} and~\ref{eq:lem-dehn-a}, which show that the action of the Dehn twist generators by cluster transformations on $\Lbb^q_{X_7}$ is intertwined with the formulas~\eqref{eq:twist-a}, \eqref{eq:twist-b}, and~\eqref{eq:twists-ab} defined in~\cite{AS19}.
\end{proof}

\section{Macdonald polynomials in genera 1 and 2}

In this section we use cluster mutations to express Macdonald operators in genera 1 and 2 via Whittaker polynomials. In particular, this yields an explicit formula for genus 2 Macdonald polynomials.

\subsection{Genus 1 Macdonald polynomials.}
\label{sec:Macdo-gen1}

We now use cluster theory to construct a basis of eigenfunctions for the Macdonald operator~\eqref{eq:Macdo-op}. The strategy is straightforward: the mutation isomorphism~\eqref{eq:discrete-mut} allows us to replace the spectral problem for the operator
$$
\eta_0(M) = t^{-1}U^{-1} + tU - tV^{-2}(U-U^{-1})(q^{-2}U^2-1)
$$
with that for simpler difference operator
\beq
\label{eq:eta-M}
\eta_{-1}(M) = t^{-1}U^{-1}+tU-t V^{-2}U^{-1}.
\eeq
Its eigenfunction equation in $\Fc$ reads
$$
\hr{tq^n + t^{-1}q^{-n}} f_l(n) - q^{-n-2} t f_l(n+2) = \big(tq^l + t^{-1}q^{-l}) f_l(n)
$$
for $l \in \Z_{\ge 0}$, and can be easily solved: the function $f_l(n)$ is zero unless $0 \le n \le l$ and $l-n \in 2\Z$, and takes the following values otherwise:
\beq
\label{eq:f_l-eigen}
f_l(n) = \frac{(q^2;q^2)_l \hr{qt^{-1}}^{n-l}}{ \big(t^2q^{n+l};q^2\big)_{\frac{l-n}{2}} \big(q^{n-l};q^2\big)_{\frac{l-n}{2}}}.
\eeq
The normalization constant $(q^2;q^2)_l$ here will ensure that the resulting Macdonald polynomial is monic. Now it follows from~\eqref{eq:intertwining} that the function
$$
f'_l(n) = \Psi^{+}_{q}[n]f_l(n) \in i_!\mathcal{F}
$$
is an eigenfunction of the operator $\eta_0(M)$ with the same eigenvalues. Hence the eigenfunctions of the Macdonald operator~\eqref{eq:Macdo-op} are given by
$$
p_l(z) = \sum_{n=0}^{\hs{\frac{l}{2}}} f'_l(l-2n)W_{l-2n}(z).
$$
Recalling the formula~\eqref{eq:Whit-poly} for the Whittaker polynomials, we have
$$
p_l(z) = \sum_{n=0}^{\hs{\frac{l}{2}}} \sum_{s=0}^{l-2n} \hr{q^{-1}t}^{2n} \frac{(q^{2(l-2n-s+1)};q^2)_{s+2n}}{\hr{q^{2(l-n)}t^2;q^2}_n \hr{q^{-2n};q^2}_n (q^2;q^2)_s} z^{l-2n-2s}.
$$
Setting $r=n+s$ and splitting three out of four Pochhammer symbols into products of two, we obtain
$$
p_l(z) = \sum_{r=0}^l \sum_{s=0}^l \hr{qt^{-1}}^{2s} \frac{(q^{2(l-m)}t^2;q^2)_s (q^{-2m};q^2)_s}{(q^{2(l-2m+1)};q^2)_s (q^2;q^2)_s} \hr{q^{-1}t}^{2r} \frac{(q^{2(l-2r+1)};q^2)_{2r}}{(q^{2(l-r)}t^2;q^2)_r (q^{-2r};q^2)_r} z^{l-2r}.
$$
The sum over $s$ in the above formula is equal to the ratio
$$
\frac{(q^{2-2r}t^{-2};q^2)_r}{(q^{2(l-2r+1)};q^2)_r}
$$
thanks to identity~\eqref{eq:gauss-thm}. Plugging the latter into the formula for $p_l$ we arrive at
$$
p_l(z) = \sum_{r=0}^l \frac{(q^{2l};q^{-2})_r (t^2;q^2)_r}{(q^{2(l-1)}t^2;q^{-2})_r (q^2;q^2)_r} z^{l-2r},
$$
which coincides with $P_l(z; t^2,q^2)$.

\subsection{Genus 2 Macdonald polynomials.}

We now use the cluster description of the algebra $\SH_{g=2}$ to derive a non-recursive formula for the genus 2 Macdonald polynomials $\Phi_{\bs l}$. The formula computes the coefficient of each monomial appearing in $\Phi_{\bs l}$ as a weighted sum over lattice points in a certain convex polytope in $\mathbb{R}^7$. Its derivation illustrates the principle that knowing the action of the mapping class group by cluster transformations allows us to reduce questions about the Macdonald-type polynomials (associated to $B$-cycles) to the corresponding ones for Whittaker-type polynomials (associated to the $A$-cycles).

Let us spell out our strategy in more detail. Tensor cube of the isomorphism
$$
\bs{\widetilde W} = \bs W \circ \mu_1 \colon \Vc/\Vc_+ \longra \Sc_{q,t},
$$
where $\bs W$ and $\mu_1$ are given by~\eqref{eq:whittaker-iso} and~\eqref{eq:discrete-mut} respectively, intertwines the operators $\eta_{-1}^{\otimes 3}(\Oc_{A_k})$ and $\Oc_{A_k}$. Now recall the mapping class group element $\gamma$ defined by~\eqref{eq:gamma-elt}. The automorphism part of the corresponding composite of cluster mutations consists of conjugation by
$$
\Psi_\gamma = \Psi_q(X_{-e_1-e_3-e_5-2e_7})\Psi_q(X_{-e_5-e_7})\Psi_q(X_{-e_3-e_7})\Psi_q(X_{-e_1-e_7})\Psi_q(X_{-e_7}).
$$
Thus, we shall first find an eigenbasis $g_{\bs l} \in \hr{\Vc/\Vc_+}^{\otimes 3}$ of the operators
$$
\Xi_{A_k} = \Ad_{\Psi_\gamma} \hr{\eta_{-1}^{\otimes 3}(\Oc_{A_k})}.
$$
Then the eigenbasis of the operators $\Oc_{A_k}$ in the space $\Sc_{q,t}^{\otimes 3}$ will be given by
\beq
\label{eq:eigen-OA}
\phi_{\bs l} = \bs{\widetilde W}(\Psi_\gamma^{-1} g_{\bs l}).
\eeq

Before we proceed, let us fix some useful notations. Given a vector $\bs n \in \Z^3$, we set
$$
n_{ij} = \frac{n_i+n_j-n_k}{2}.
$$
We then define the vector $\bs n' \in \Z^3$ by
$$
n'_k = n_{ij}
$$
for $(i,j,k)$ being a permutation of $(1,2,3)$, and note that
$$
n_k = n'_i + n'_j.
$$
Recall that $\hc{\bs{\delta^j}}$ denotes the standard basis in $\Z^3$, and define $\Omega$ to be the $3$-by-$3$ matrix with columns $\big(\bs{\delta^1}\big)',\big(\bs{\delta^2}\big)',\big(\bs{\delta^3}\big)'$, i.e. the matrix such that
$$
\Omega \bs{\delta^k} = \big(\bs{\delta^k}\big)'
$$
for $k=1,2,3$. Recall that a triple of non-negative integers $\bs j$ is \emph{admissible} if
$$
\Omega{\bs j} \in \Z_{\ge0}^3,
$$
which is equivalent to the conditions that $\bs j \in 2\Z$ and components of $\bs j$ satisfy the triangle inequality. Then the linear transformation $\Omega$ identifies $\mathcal{H}^{\otimes 3}$ with the space of all compactly supported functions $g(\bs j)$  on the lattice
$$
\hc{\bs j\in\Z^3 \,\big|\, \underline{\bs j} \in 2\Z}.
$$  

Under the above identification, the operators $\Xi_{A_k}$ take form
$$
\Xi_{A_k} = t^{-1}U_k^{-1}+t U_k-V_k^{-2}U_k^{-1},
$$
which only differs from~\eqref{eq:eta-M} by the absence of the factor $t$ in the third summand. The eigenfunction equations for the operators $\Xi_{A_k}$ on $\Vc/\Vc_+$ then read
$$
\big(tq^{l_k} + t^{-1}q^{-l_k}\big) g_{\bs l}(\bs j) = \left(t q^{j_k} + t^{-1}q^{-j_k} \right) g(\bs j) - q^{-(j_k+2)}g(\bs j + 2\bs{\delta^k}).
$$
The eigenfunctions $g_{\bs l}$ are labelled by admissible triples~$\bs l$, and can be easily computed: they are zero unless $\bs j\in\Z_{\ge0}^3$ and $\frac12(\bs l - \bs j) \in \Z_{\ge0}^3$, and have nonzero values given by\footnote{Note that after dividing the formula~\eqref{eq:easy-eigenfn} by the normalization factor~\eqref{eq:norm-factor} it becomes very similar to a triple product of~\eqref{eq:f_l-eigen}.}
\beq
\label{eq:easy-eigenfn}
g_{\bs l }(\bs j) = 
\Psi_q(-qt^2) (q^2t^{-1})^{\frac12(\underline{\bs j}+\underline{\bs l})} \prod_{k=1}^3(q^{-2l_k};q^2)_{\frac{1}{2}(j_k+l_k)}(t^2;q^2)_{\frac{1}{2}(j_k+l_k)}.
\eeq
Our goal now is to compute the eigenfunctions~\eqref{eq:eigen-OA} of the operators  $\Oc_{A_k}$ in $\Sc_{t,q}$, which take the form
$$
\phi_{\bs l}(\bs x) = \sum_{\bs j} (\Psi_\gamma^{-1} g_{\bs l})(\bs j) \widetilde W_{\bs{j'}}(\bs x),
$$
with
$$
\widetilde W_{\bs{j'}}(\bs x) = \frac{W_{\bs{j'}}(\bs x)}{(q^2;q^2)_{\bs{j'}}}.
$$

Applying the first factor $\Psi_q(X_{-e_1-e_3-e_5-2e_7})^{-1}$ of $\Psi_\gamma^{-1}$ to the function $g_{\bs l}$, we obtain
\begin{align*}
g'_{\bs l }(\bs j) = \Psi_q(-q^{\underline{\bs j}+1}t^2)^{-1}g_{\bs l}(\bs j)
= \frac{(q^2t^{-1})^{\frac{1}{2}(\underline{\bs j}+\underline{\bs l})} }{(q^2t^2;q^2)_{\frac{1}{2}\underline{\bs j}}} \prod_{k=1}^3 (q^{-2l_k};q^2)_{\frac{1}{2}(j_k+l_k)}(t^2;q^2)_{\frac{1}{2}(j_k+l_k)}.
\end{align*}
Next we apply the product of three commuting operators $\prod_{k=1}^3 \Psi_q(X_{-e_{2k-1}-e_7})^{-1}$. Recalling the Taylor series for the quantum dilogarithm:
$$
\Psi_q^{-1}(x) = \sum_{n\ge0} \frac{q^{-n}}{(q^{-2n};q^2)_n}x^n,
$$
we see that the action of each factor on a compactly supported function $f$ is given by
$$
(\Psi_q^{-1}(X_{-e_{2a-1}-e_7})\cdot f)(\bs j) = \sum_{n\ge 0}\frac{q^{2n j'_a}t^n}{(q^2;q^2)_n} f(\bs j+2n \bs{\delta^a}) .
$$
By the vanishing property of $g_{\bs l}$, we get
\begin{align*}
g_{\bs l}''(\bs j) &=  \sum_{n_a=0}^{\frac{1}{2}(l_a-j_a)} t^{\underline{\bs n}} q^{\underline{\bs n}^2 - \hm{\bs n}^2 + 2 \bs n \cdot \bs{j'}} \prod_{a=1}^3(q^2;q^2)_{n_a}^{-1} g_{\bs l}'(\bs j+2\bs n) \\
&= \hr{q^2t^{-1}}^{\frac12(\underline{\bs j}+\underline{\bs l})} \sum_{n_a=0}^{\frac{1}{2}(l_a-j_a)}
\frac{q^{\underline{\bs n}^2 - \hm{\bs n}^2 + 2 \bs n \cdot \bs{j'}+2\underline{\bs n}}}{(q^2t^2;q^2)_{\frac{1}{2}\underline{\bs j}+\underline{\bs n}}}
\cdot \prod_{a=1}^3\frac{(q^{-2l_a};q^2)_{\frac{1}{2}(j_a+l_a)+n_a}(t^2;q^2)_{\frac{1}{2}(j_a+l_a)+n_a}}{(q^2;q^2)_{n_a}}.
\end{align*}
Then applying the final mutation $\Psi_q(X_{-e_7})^{-1}$, we have
$$
(\Psi_\gamma^{-1}g_{\bs l})(\bs j) = \sum_{s\ge0}(-1)^s \frac{(q^{-2}t)^s}{(q^{-2s};q^2)_s} g_{\bs l}''(\bs j+2s(1,1,1)).
$$
Recalling the formula~\eqref{eq:Whit-poly} for the Whittaker polynomials $W_l$, and collecting coefficients of each monomial, we arrive at
\begin{theorem}
\label{thm:g2-eigen}
For each admissible triple $\bs l$ and a triple of non-negative integers $\bs k \in \Z_{\ge 0}^3$ satisfying $\underline{\bs k}\le  \underline{\bs l}$, consider the convex polytope in the positive orthant of 7-dimensional space
$$
\Delta\hr{\bs k | \bs l} = \{(r_{23},r_{13},r_{12},s,n_1,n_2,n_3)\} \subset \mathbb{Z}_{\ge0}^7
$$
given by the inequalities
$$
2k_{ab} - l_{ab} \le r_{ab} \qquad\text{and}\qquad r_{ab} + r_{ac} \le n_a \le k_a-s.
$$
Define the rational function $C_{\bs l,\bs m}(\bs{r'},s,\bs n) \in\mathbb{K}$ by 
\begin{align*}
C_{\bs l,\bs k}(\bs{r'},s,\bs n) &= \frac{(-1)^s(q^2t^{-1})^{\underline{\bs l}+\underline{\bs r}-\underline{\bs k} +2s }q^{\underline{\bs n}^2 - \underline{\bs r}^2 + 3\hm{\bs r}^2 - \hm{\bs n}^2 - 2\bs r \cdot \bs n+(s+1)(s+2(\underline{\bs n} - \underline{\bs r}))}}{(q^2;q^2)_s(q^2t^2;q^2)_{\underline{\bs n} + \frac{1}{2}\underline{\bs l} -\underline{\bs k} +3s }}\\
&\times \prod_{a=1}^3 q^{2(n_a-r_a)(l'_a-2k'_a)} \frac{(t^2;q^2)_{n_a+l_a-k_a+s}(q^{-2l_a};q^2)_{n_a+l_a-k_a+s}}{(q^2;q^2)_{n_a-r_a}(q^2;q^2)_{r'_a}(q^2;q^2)_{l'_a+r'_a-2k'_a}}.
\end{align*}
Then the polynomial
$$
\phi_{\bs l}(\bs x) = \sum_{\underline{\bs k} \le \underline{\bs l}}\sum_{(\bs{r'},s,\bs{n})\in \Delta(\bs{k}|\bs{l})}C_{\bs{l},\bs{k}}(\bs{r'},s,\bs{n})\prod_{1 \le a<b \le b}x_{ab}^{l_{ab}-2k_{ab}}
$$
is a joint eigenfunction of the difference operators $\mathcal{O}_{A_{k}}$ with eigenvalues $\left(t q^{l_k} + t^{-1}q^{-l_{k}} \right)$.
\end{theorem}

When $\bs{k}=0$, the polytope $\Delta(0|\bs{l})$ consists of the single point ${0}\in\mathbb{Z}^7$. Hence we obtain

\begin{cor}
The coefficient $K_{0}(\bs{l})$ of the monomial $x_{23}^{l_{23}}x_{13}^{l_{13}}x_{12}^{l_{12}}$ in $\phi_{\bs l} $ is 
\beq
\label{eq:norm-factor}
K_{0}(\bs{l}) = \frac{t^{-\underline{\bs l}}q^{2\underline{\bs l}}}{(q^2t^2;q^2)_{\frac{1}{2}\underline{\bs l}}}\prod_{a=1}^3\frac{(t^2;q^2)_{l_a}(q^{-2l_a};q^2)_{l_a}}{(q^2;q^2)_{l'_{a}}}.
\eeq
\end{cor}

To relate this normalization to the one used in~\cite{AS19}, we need to understand the Pieri rules for the $\phi_{\bs l}$. This too we can easily work out using cluster transformations.

\begin{theorem}
The eigenfunctions $\phi_{\bs l} $ satisfy the Pieri rule
\beq
\label{eq:norm-pieri}
(x_{ij}+x_{ij}^{-1})\phi_{\bs l} = \frac{\hr{1-q^{2l_i}}\hr{1-q^{2l_j}}}{\hr{1-t^2q^{2l_i}}\hr{1-t^2q^{2l_j}}} \sum_{a,b\in \{\pm1\}}\widetilde{A}_{a,b} \phi_{\bs l+a\bs{\delta^i}+b\bs{\delta^j}},
\eeq
where
\begin{align*}
\widetilde{A}_{+,+} &= t^2q^{2(l_i+l_j)}\frac{\left(1-t^2q^{\underline{\bs l} + 2}\right)\left(1-q^{(2(l_{ij}+1))}\right)}{\big(1-q^{2l_i}\big)\big(1-q^{2l_j}\big)\left(1-q^{2(l_i+1)}\right)\left(1-q^{2(l_j+1)}\right)}, \\
\widetilde{A}_{+,-} &=t q^{2(l_i-l'_i+1)} \frac{\big(1-t^2q^{2(l'_i-1)}\big)\big(1-q^{2(l'_j-1)}\big)}{\big(1-q^{2l_i}\big)\big(1-q^{2(l_i+1)}\big)} , \\
\widetilde{A}_{-,+} &= t q^{2(l_j-l'_j+1)} \frac{\big(1-t^2q^{2(l'_j-1)}\big) \big(1-q^{2(l'_i-1)}\big)}{\big(1-q^{2l_j}\big)\big(1-q^{2(l_j+1)}\big)}, \\
\widetilde{A}_{-,-} &= t^{-2}q^{2(2-l_i-l_j)}{\big(1-t^4q^{\underline{\bs l} - 2}\big)\big(1-t^2q^{2(l_{ij}-1)}\big)}.\phantom{\frac{q^{2(l'_j-1)}}{q^{2(l_j+1)}}}
\end{align*}
\end{theorem}

\begin{proof}
We give the proof for $i=2$ and $j=3$, the other two cases are identical. It follows from~\eqref{eq:B-cycles} that at the level of the expansion coefficients with respect to the basis $\widetilde W_{\bs j}$, the Pieri rule~\eqref{eq:norm-pieri} is equivalent to the identity
\beq
\label{eq:Pieri-Psi-g}
\left(V_{23}^{-1}+(1-U_{23}^2)V_{23}\right)\cdot (\Psi_\gamma^{-1}g_{\bs l}) = \sum_{a,b}~\widetilde{A}_{a,b}\Psi_\gamma^{-1}g_{\bs l + a\bs{\delta^2} + b\bs{\delta^3}}.
\eeq
This is not so difficult to check using the explicit formula for the coefficient $ \Psi_\gamma^{-1}g_{\bs l}$ above.
Alternatively, in view of the intertwining relation
$$
\left(V_{23}^{-1}+(1-U_{23}^2)V_{23}\right)\circ \Psi_\gamma^{-1} = \Psi_\gamma^{-1} \circ Z_{23},
$$
where
$$
Z_{23} = V_{23}^{-1}+(1-t^2U_{12}^2U_{13}^2U_{23}^2)(1-U_{23}^2)V_{23} + q^{-2}t U_{23}^2V_{12}^{-1}V_{13}^{-1} \hr{V_{23}^2 - U_{12}^2V_{12}^2 - U_{13}^2V_{13}^2}
$$
we can translate the identity~\eqref{eq:Pieri-Psi-g} for $\Psi_\gamma^{-1}g_{\bs l}$ into the following identity for the simpler eigenfunctions $g_{\bs l}$ defined by~\eqref{eq:easy-eigenfn}:
\begin{multline*}
\sum_{a,b \in \hc{\pm}}{\widetilde B}_{a,b} g_{\bs l}(\bs j+a\bs{\delta^2}+b\bs{\delta^3}) + q^{-2}t q^{2j_{23}}g_{\bs l}(\bs j+2\bs{\delta^1}-\bs{\delta^2}+\bs{\delta^3}) = \sum_{a,b \in \hc{\pm}}\widetilde{A}_{a,b}g_{\bs l+a\bs{\delta^2}+b\bs{\delta^3}}(\bs j),
\end{multline*}
where
$$
\widetilde B_{+,+} = 1, \qquad
\widetilde B_{+,-} = -t q^{2j_3}, \qquad
\widetilde B_{-,+} = -t q^{2j_2}, \qquad
\widetilde B_{-,-} = (1-t^2q^{\underline{\bs j}})(1-q^{2j_{23}}).
$$
The latter identity is straightforward to verify using the functional equation for the $q^2$-Pochhammer symbol.
\end{proof}

From this we easily deduce the relation between the two normalizations. Indeed, 
setting
$$
N_{X_7}(\bs l)=q^{\hm{\bs l}^2} (t^2q^2;q^2)_{\frac{1}{2}\underline{\bs l}} \prod_{a=1}^3 \frac{(q^2;q^2)_{l'_a}}{(q^2t^2;q^2)_{l_a}},
$$
it follows that the $A_{+,+}$-term in the Pieri rules for the renormalization $N_{X_7}(\bs l) \cdot \phi_{\bs l}$ becomes equal to $q^{-2}t^2$. We can do a similar thing for the Pieri rules in~\cite{AS19}, and this tells us the ratio between the two normalizations: setting
$$
N_{AS}(\bs l) = (t^4;q^2)_{\frac{1}{2}\underline{\bs l}}^{-1} \prod_{a=1}^3 \frac{(t^4;q^2)_{l_a}}{(t^2;q^2)_{l'_a}}
$$
brings the coefficient $A_{+,+}$ for the basis $N_{AS}(\bs l)\Phi_{\bs l}$ to $t$. Thus we conclude that
$$
\Phi_{\bs l} = (tq^{-2})^{\frac{1}{2}\bs l}\frac{N_{X_7}(\bs l) }{N_{AS}(l) }\phi_{\bs l}.
$$

\section{Analytic theory of the genus 2 DAHA}
\label{sec:analytic}
The cluster realization of the algebra $\mathbb{SH}_{g=2}$ provided by Theorem~\ref{thm:g2-embed} allows one to define an analytic analog of its representation by difference operators on the space of symmetric polynomials. Indeed, by the general construction of~\cite{FG09}, the universally Laurent algebra $\widehat \Lbb^q_{\bs{X_7}}$ has a family of \emph{positive} representations parametrized by two real numbers: the Planck's constant $\hbar\in\mathbb{R}$, related to $q$ via $q =e^{\pi i \hbar^2}$, together with a real number $\tau\in\mathbb{R}$ which determines the character by which the centre of $\widehat \Lbb^q_{\bs{X_7}}$ acts. In more detail, the underlying linear space of these representations is the dense \emph{Fock-Goncharov Schwartz space} $\mathcal{S}\subset \Hc$ inside a Hilbert space $\Hc\simeq L_2(\mathbb{R}^3, d\bs x)$. The cluster modular group, and hence the mapping class group of~$\Sigma_{2,0}$, acts on this Hilbert space by unitary intertwiners. The Schwartz space $\mathcal{S}$ carries an action not only of $\widehat\Lbb^q_{\bs{X_7}}$, but also of its \emph{modular double} 
$$
\widehat \Lbb^{q, \tilde q}_{\bs{X_7}} = \widehat \Lbb^q_{\bs{X_7}} \otimes_{\mathbb{C}} \widehat \Lbb^{\tilde q}_{\bs{X_7}}.
$$
Here we set
$$
\quad q = e^{\pi i\hbar^2} \qquad\text{and}\qquad \tilde q = e^{\pi i\hbar^{-2}},
$$
so that the quantum parameters for the two factors are related by the modular transformation $\hbar\mapsto 1/\hbar$.
The analytic theory of quantum cluster varieties thus provides us with a natural representation of the modular double of $\mathbb{SH}_{g=2}$.  

In this context, one can consider the spectral problem for the commuting operators $\mathcal{O}_{A_i}$, and attempt to construct a unitary joint eigenfunction transform for them. In the genus 1 case, this program was carried out in the paper~\cite{DFK+24}, where the eigenfunctions were identified with matrix coefficients of the mapping class group element
$$
S= \begin{pmatrix} 0&-1\\1&0\end{pmatrix}\in SL(2;\mathbb{Z}).
$$
Similarly to the genus 1 case, one can present the genus 2 Macdonald eigenfunction as a matrix coefficient of the mapping class $\sigma$, defined in~\eqref{eq:sigma-elt}, and we expect this description to shed light on the symmetries and bispectral properties of the genus 2 Macdonald functions, see e.g.~\cite{DFK23}. We hope to return to this aspect of the analytic theory of $\mathbb{SH}_{g=2}$ on a future occasion.

\section{Semi-classical limit}
\label{sec:class}

In this section we recall the main constructions and results of \cite{CS23} in order to connect the algebra $\widehat\Lbb_{X_7}^q$ to several well-known Poisson manifolds. First, let us briefly recall the setup of cluster Poisson varieties, see~\cite{FG06}. A quiver $Q$ determines a toric chart
$$
\Tc_Q = \Spec\big(\Tc_Q^q\big|_{q=1}\big).
$$
with a Poisson bracket defined on the natural toric coordinates by
$$
\hc{y_j,y_k} = \eps_{jk} y_j y_k.
$$
Cluster mutations define the gluing data between pairs of ``neighboring'' charts, and are given by the $q=1$ specialization of the quantum ones. Similarly, the classical universally Laurent algebra $\widehat\Lbb_{\bs Q}$ is the $q=1$ specialization of the quantum universally Laurent algebra $\widehat\Lbb_{\bs Q}^q$. The cluster Poisson variety $\Xc_{\bs Q}$ is then defined as
$$
\Xc_{\bs Q} = \Spec\hr{\widehat\Lbb_{\bs Q}}.
$$
Since the formulas for cluster mutations are subtraction-free, it makes sense to talk about the \emph{positive part} $\Xc_{\bs Q}^+$ of the cluster variety $\Xc_{\bs Q}$, defined by the condition that the cluster coordinates in any, hence in all, cluster charts take real positive values.

In the case $Q=X_7$, the cluster Poisson variety $\Xc_{\bs{X_7}}$ is equipped with a Poisson bivector field of corank 1. A Casimir function generating the Poisson centre of $\widehat \Lbb_{X_7}$ can be chosen as
$$
C = y_7 \cdot \prod_{i=1}^6 \sqrt{y_i}
$$
in the variables of the initial cluster. The main result of \cite{CS23} is a construction of a surjective Poisson map
$$
\kappa \colon \mathbb V^+_{\bs{X_7}}(C-1) \longra \Tc_{2,0}
$$
from the totally positive part of the subvariety $\mathbb V_{\bs{X_7}}(C-1)$ cut out of $\Xc_{\bs{X_7}}$ by the equation $C=1$, onto the Teichm\"uller space $\Tc_{2,0}$ of hyperbolic metrics on a closed surface of genus 2. The map $\kappa$ is not bijective, but has finite fibers.

The isomorphism $\kappa$ is derived from constructing global log-canonical coordinates on the subgroup $U \subset GL_3(\R)$ of unipotent 3-by-3 upper triangular matrices, equipped with the structure of a symplectic groupoid. The objects of the \emph{symplectic groupoid} $\Mc$ are elements $A \in U$, and the morphisms are pairs $(B,A) \in GL_3 \times U$, such that $A' = BAB^t \in U$. The groupoid $\mathcal M$ is equipped with a canonical symplectic form, see~\cite{Wei88}. The push-forward of the dual nondegenerate Poisson bracket determines a natural Poisson bracket on $U$, which was studied in~\cite{Dub96, Dub99, Uga99} in the context of Frobenius manifolds and isomonodromic deformations.

Let $\widetilde{GL}_3$ denote a symplectic leaf of maximal dimension in the group $GL_3$ endowed with the standard Poisson--Lie structure. For any  $B\in\widetilde{GL}_3$ there exists a unique $A\in U$, such that the pair $(B,A)$ is a morphism in $\Mc$. In this way we obtain a Poisson map
$$
\eta \colon \WGL_3 \longra U \times U, \qquad B \longmapsto (A,A').
$$
Its image coincides with the subvariety of $U^{\times 2}$ cut out by the equation
$$
\Mgt(A) = \Mgt(A'),
$$
where $\Mgt$ is the Markov function on $U$, defined via
$$
\Mgt(A) = \det(A+A^t).
$$
Now consider the quiver $X_6$ shown on Figure~\ref{fig:X6}. As was shown in~\cite{CS23}, both the source and the target of the map $\eta$ admit Poisson maps from the cluster chart $\Tc_{X_6}$. This in turn yields the following commutative diagram:
$$
\begin{tikzcd}
\WGL_3 \arrow{rr}{\eta} && U^{\times 2} \\
& \Tc_{X_6} \arrow{lu}{\alpha} \arrow[swap]{ru}{\beta}
\end{tikzcd}
$$
where the map $\alpha \colon \Tc_{X_6} \to \WGL_3$ is surjective, and the image of $\beta \colon \Tc_{X_6} \to U^{\times 2}$ coincides with that of $\eta$.

\begin{figure}[h]
\begin{tikzpicture}[every node/.style={inner sep=0, minimum size=0.45cm, thick, draw, circle}, thick, x=1cm, y=0.866cm]

\node (v1) at (-2,0) {\footnotesize{1}};
\node (v2) at (-1,2) {\footnotesize{2}};
\node (v3) at (1,2) {\footnotesize{3}};
\node (v4) at (2,0) {\footnotesize{4}};
\node (v5) at (0,-2) {\footnotesize{5}};
\node (v6) at (0,0) {\footnotesize{6}};

\draw [->] (v6) to (v2);
\draw [->] (2.-105) to (1.45);
\draw [->] (2.-135) to (1.75);
\draw [->] (v1) to (v6);
\draw [->] (v6) to (v4);
\draw [->] (4.105) to (3.-45);
\draw [->] (4.135) to (3.-75);
\draw [->] (v3) to (v6);
\draw [->] (v5) to (v6);

\end{tikzpicture}
\caption{The quiver $X_6$.}
\label{fig:X6}
\end{figure}
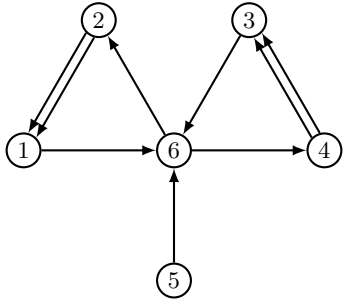

In order to relate the above commutative diagram to Teichm\"uller spaces, let us recall a Poisson map $\rho \colon \Tc_{1,1} \to U$ constructed in~\cite{CF00}, where $\Tc_{1,1}$ is the Teichm\"uller space of genus one hyperbolic surfaces with one hole, equipped with the Goldman Poisson bracket. Then we have
$$
\rho^* \colon \Mgt \longmapsto 2\cosh(\ell/2),
$$
where $\ell$ is the hyperbolic length of the boundary of the hole. On the other hand, the operation of cutting $\Gamma_{2,0}$ along the separating curve labelled $\Mgt$ in Figure~\ref{fig:genus2curve} induces a Poisson map
$$
\xi \colon \Tc_{2,0} \longra \Tc_{1,1}\times \Tc_{1,1}
$$
whose image is cut out by the equation $\cosh(\ell_1/2) = \cosh(\ell_2/2)$. Then a Poisson surjection
$$
\iota \colon \Tc_{X_6}^+ \longra \Tc_{2,0}
$$
was constructed in~\cite{CS23}, making the following diagram of Poisson maps commutative:
$$
\begin{tikzcd}
\Tc_{1,1}^{\times 2} \arrow{r}{\rho^{\times 2}} & U^{\times 2} \\
\Tc_{2,0} \arrow{u}{\xi} & \Tc_{X_6}^+ \arrow{l}{\iota} \arrow[swap]{u}{\beta}
\end{tikzcd}
$$
Upon an attempt to make the map $\iota$ into a $\Gamma_{2,0}$-equivariant map $\Xc_{\bs{X_6}}^+ \to \Tc_{2,0}$, where $\Gamma_{2,0}$ is the mapping class group of $\Sigma_{2,0}$, it was discovered that the Dehn twist along a cycle crossing the separating curve $\Mgt$, see Figure~\ref{fig:genus2curve}, is not realized as a cluster transformation. This drawback was, however, remedied by the Theorem~\ref{thm:genus-2}, which constitutes the main result of~\cite{CS23}.\footnote{Which also happens to be Theorem 7.1 in \emph{loc.cit.}}

In what follows, we denote by $G_\gamma \in \Oc(\Tc_{2,0})$ the geodesic length of an element $\gamma \in \pi_1(\Sigma_{2,0})$. Then the elements $G_{B_{ij}}$ for $1 \le i \le j \le 3$ along with $G_{A_1}$ and $G_{A_3}$, see Figure~\ref{fig:genus2curve}, generate $\C(\Tc_{2,0})$ as a Poisson algebra. Namely, any element $G_\gamma$ may be expressed through them via the skein relation
$$
G_\alpha G_\beta = G_{\alpha\beta} + G_{\alpha^{-1}\beta}
$$
and the Goldman Poisson bracket
$$
\hc{G_\alpha, G_\beta} = \frac12\hr{G_{\alpha\beta} - G_{\alpha^{-1}\beta}},
$$
both of which hold for any $\alpha, \beta \in \pi_1(\Sigma_{2,0})$ such that $\hm{\alpha \cap \beta} = 1$. 

\begin{theorem}[\cite{CS23}]
\label{thm:genus-2}
The mapping class group $\Gamma_{2,0}$ acts on $\Xc_{\bs{X_7}}$ via cluster transformations and preserves the locus $\mathbb V^+_{\bs{X_7}}(C-1)$. Moreover, there exists a $\Gamma_{2,0}$-equivariant finite Poisson cover
$$
\kappa \colon \mathbb V^+_{\bs{X_7}}(C-1) \longra \Tc_{2,0},
$$
such that the map $\kappa^* \colon \Oc(\Tc_{2,0}) \to \widehat\Lbb_{\bs{X_7}}/\ha{C-1}$ reads
\beq
\label{eq:GB}
\begin{aligned}
G_{B_{12}} &\longmapsto (y_5y_6)^{\frac12} + (y_6/y_5)^{\frac12} + (y_5y_6)^{-\frac12}, \\
G_{B_{13}} &\longmapsto (y_3y_4)^{\frac12} + (y_4/y_3)^{\frac12} + (y_3y_4)^{-\frac12}, \\
G_{B_{23}} &\longmapsto (y_1y_2)^{\frac12} + (y_2/y_1)^{\frac12} + (y_1y_2)^{-\frac12},
\end{aligned}
\eeq
and
\beq
\label{eq:GA}
\begin{aligned}
G_{A_1} &\longmapsto y_7(y_2y_1y_3y_5)^{\frac12} + G_{B_{23}}(y_3y_5)^{\frac12} + y_7^{-1}(y_2y_1y_3y_5)^{-\frac12}(1+y_3)(1+y_5), \\
G_{A_3} &\longmapsto y_7(y_6y_1y_3y_5)^{\frac12} + G_{B_{12}}(y_1y_3)^{\frac12} + y_7^{-1}(y_6y_1y_3y_5)^{-\frac12}(1+y_1)(1+y_3).
\end{aligned}
\eeq
\end{theorem}

It only remains to notice that the formulas~\eqref{eq:GB} and~\eqref{eq:GA} are the $q=1$ specializations of the formulas~\eqref{eq:B-cycles} and~\eqref{eq:A-cycles} respectively.

\bibliographystyle{alpha}

\end{document}